\documentclass{amsart}
\author{Dilip Raghavan}
\thanks{First author partially supported by National University of Singapore 
research grant number R-146-000-161-133.}
\address{Department of Mathematics \\
National University of Singapore\\
Singapore 119076}
\email{raghavan@math.nus.edu.sg}
\urladdr{http://www.math.toronto.edu/raghavan}
\author{Saharon Shelah}
\date{\today}
\title[On embedding certain partial orders]{On embedding certain partial orders into the P-points under RK  and Tukey reducibility}
\usepackage{amssymb, amsmath, amsthm, mathrsfs, enumerate, amsfonts, latexsym, bbm}
\def\polhk#1{\setbox0=\hbox{#1}{\ooalign{\hidewidth
    \lower1.5ex\hbox{`}\hidewidth\crcr\unhbox0}}}
\newtheorem{Theorem}{Theorem}
\newtheorem{Claim}[Theorem]{Claim}
\newtheorem{Lemma}[Theorem]{Lemma}
\newtheorem{Cor}[Theorem]{Corollary}

\newtheorem{Question}[Theorem]{Question}
\newtheorem*{mth}{Main Theorem}
\theoremstyle{definition}
\newtheorem{Def}[Theorem]{Definition}

\theoremstyle{remark}
\newtheorem{remark}[Theorem]{Remark}

\newcommand{\restrict}{\upharpoonright}
\newcommand{\forallbutfin}{{\forall}^{\infty}}
\newcommand{\existsinf}{{\exists}^{\infty}}
\renewcommand{\c}{\mathfrak{c}}

\newcommand{\p}{{\mathfrak{p}}}

\renewcommand{\[}{\left[}
\renewcommand{\]}{\right]}
\renewcommand{\P}{\mathbb{P}}

\newcommand{\Q}{\mathbb{Q}}
\newcommand{\lc}{\left|}
\newcommand{\rc}{\right|}

\newcommand\ZFC{\mathrm{ZFC}}
\newcommand\FIN{\mathrm{FIN}}
\newcommand\MA{\mathrm{MA}}
\newcommand\PFA{\mathrm{PFA}} 

\newcommand\CH{\mathrm{CH}}
   
\newcommand{\BS}{{\omega}^{\omega}}

\DeclareMathOperator{\cre}{\mathrm{\mathcal{C}\mathcal{R}}}
\DeclareMathOperator{\nor}{nor}

\DeclareMathOperator{\cf}{cf}

\newcommand{\Pset}{\mathcal{P}}

\newcommand{\BB}{\mathcal{B}}
\newcommand{\A}{{\mathscr{A}}}
\newcommand{\C}{{\mathscr{C}}}
\newcommand{\D}{{\mathscr{D}}}
\newcommand{\DD}{\mathcal{D}}

\newcommand{\GG}{{\mathcal{G}}}

\newcommand{\U}{{\mathcal{U}}}
\newcommand{\VV}{{\mathcal{V}}}
\newcommand{\cube}{{\[\omega\]}^{\omega}}

\newcommand{\F}{{\mathcal{F}}}

\newcommand{\T}{{\mathcal{T}}}

\newcommand{\Ra}{\mathbb{R}}
\newcommand{\XX}{\mathscr{X}}
\newcommand{\YY}{\mathscr{Y}}
\begin{document}
\begin{abstract}
The study of the global structure of ultrafilters on the natural numbers with respect to the quasi-orders of Rudin-Keisler and Rudin-Blass reducibility was initiated in the 1970s by Blass, Keisler, Kunen, and Rudin.
In a 1973 paper Blass studied the special class of P-points under the quasi-ordering of Rudin-Keisler reducibility.
He asked what partially ordered
sets can be embedded into the P-points when the P-points are equipped with this ordering.
This question is of most interest under some hypothesis that guarantees the existence of many P-points, such as Martin's axiom for
$\sigma$-centered posets. 
In his 1973 paper he showed under this assumption that both ${\omega}_{1}$ and the reals can be embedded.
This result was later repeated for the coarser notion of Tukey reducibility.
We prove in this paper that Martin's axiom for $\sigma$-centered posets implies that every partial order of size at most continuum can be embedded into the P-points both under Rudin-Keisler and Tukey reducibility.
\end{abstract}
\maketitle
\section{Introduction} \label{sec:intro}
The analysis of various quasi-orders on the class of all ultrafilters on $\omega$ provides a great deal of information about the global structure of this class.
An early example of such global information was the proof that $\beta \omega \setminus \omega$ is not homogeneous, obtained through an analysis of what later became known as the Rudin-Frol{\'{\i}}k order (see \cite{betannothomo}).
This ordering and the weaker Rudin-Keisler ordering were analyzed in \cite{maryellenorder} to obtain more information about the topological types in $\beta\omega \setminus \omega$.
An analysis of the stronger Rudin-Blass order eventually led to the isolation of the principle of near coherence of filters, a principle which postulates a kind of global compatibility between ultrafilters on $\omega$, and has applications to diverse areas of mathematics (see \cite{originalncf1, originalncf2, originalncf3}).
Larson~\cite{medialpaul} is a recent application of a slightly stronger principle than near coherence to measure theory. 
Recall the following definitions:
\begin{Def} \label{def:rkandrb}
Let $\F$ be a filter on a set $X$ and $\GG$ a filter on a set $Y$.
We say that $\F$ is \emph{Rudin-Keisler (RK) reducible to} $\GG$ or
\emph{Rudin-Keisler(RK) below} $\GG$, and we write $ \F \; {\leq}_{RK} \; \GG$, if there is a map $f: Y \rightarrow X$ such that for each $a \subset X$, $a \in \F$ iff ${f}^{-1}(a) \in \GG$. $\F$ and $\GG$ are \emph{RK equivalent}, written $\F \; {\equiv}_{RK} \; \GG$, if $\F \; {\leq}_{RK} \; \GG$ and $\GG \; {\leq}_{RK} \; \F$.

We say that $\F$ is \emph{Rudin-Blass (RB) reducible to} $\GG$ or \emph{Rudin-Blass (RB) below} $\GG$, and we write $ \F \; {\leq}_{RB} \; \GG$, if there is a finite-to-one map $f: Y \rightarrow X$ such that for each $a \subset X$, $a \in \F$ iff ${f}^{-1}(a) \in \GG$. RB equivalence is defined analogously to RK equivalence.
\end{Def}
In this paper we restrict ourselves only to ultrafilters on $\omega$.
If $\F$ and $\GG$ are ultrafilters on $\omega$, then $\F \; {\equiv}_{RK} \; \GG$ if and only if there is a permutation $f:\omega\to\omega$ such that
$\F = \{a\subset\omega:f^{-1}(a)\in \GG\}$.  For this reason, ultrafilters that are RK equivalent are sometimes said to be \emph{(RK) isomorphic}.
If $f: \omega \rightarrow \omega$ is a function such that $\forall b \in \GG \[f''b \in \F\]$, then in the case when $\F$ and $\GG$ are ultrafilters on $\omega$, $f$ already witnesses that $\F \; {\leq}_{RK} \; \GG$. 
 
Kunen~\cite{rknonlinear} was the first to construct two ultrafilters $\U$ and $\VV$ on $\omega$ such that $\VV \; {\nleqslant}_{RK} \; \U$ and $\U \; {\nleqslant}_{RK} \; \VV$ using only the axioms of $\ZFC$.
His techniques actually showed in $\ZFC$ alone that the class of ultrafilters on $\omega$ has a fairly complicated structure with respect to the ordering ${\leq}_{RK}$.

It is also well-known that certain special classes of ultrafilters can be characterized using the Rudin-Keisler order.
Recall the following notions.
\begin{Def} \label{def:selectiveandppoint}
An ultrafilter $\U$ on $\omega$ is \emph{selective} if, for every function $f:\omega\to\omega$, there is a set $A \in \U$ on which $f$ is either one-to-one or constant.
$\U$ is called a \emph{P-point} if, for every $f:\omega\to\omega$, there is $A \in \U$ on which $f$ is finite-to-one or constant.
\end{Def}
It is easy to see that an ultrafilter $\U$ on $\omega$ is a P-point iff for any collection $\{{a}_{n}: n \in \omega\}$ there exists $a \in \U$ such that $\forall n \in \omega\[a \; {\subset}^{\ast} \; {a}_{n}\]$.
Here ${\subset}^{\ast}$ denotes the relation of containment modulo a finite set: $a \; {\subset}^{\ast} \; b$ iff $a \setminus b$ is finite.
Selective ultrafilters are minimal in the Rudin-Keisler ordering, meaning that any ultrafilter that is RK below a selective ultrafilter is RK equivalent to that selective ultrafilter.
This minimality in fact characterizes the selective ultrafilters.
P-points are minimal in the Rudin-Frol{\'{\i}}k.
Observe that ${\leq}_{RK}$ and ${\leq}_{RB}$ coincide for the class of P-points.

Rudin~\cite{walterppoints} proved in 1956 that P-points exist if the Continuum Hypothesis ($\CH$ henceforth) is assumed, and he used this to show that $\CH$ implies the non-homogeneity of $\beta\omega\setminus\omega$.
P-points were also independently considered by several other people in a more model-theoretic context.
The question of whether P-points always exist was settled in a landmark paper of Shelah in 1977 (see \cite{PIF}), where the consistency of their non-existence was proved.

Blass considered the structure of the class of P-points in \cite{blassrk} with respect to the Rudin-Keisler order.
As the existence of P-points is independent of $\ZFC$, it makes sense to consider this structure only when some hypothesis that allows us to build P-points with ease is in hand.
If this hypothesis is relatively mild and moreover has the status of a ``quasi-axiom'', then it may be considered the ``right axiom'' under which to investigate the class of P-points.
In \cite{blassrk}, Blass used Martin's axiom for $\sigma$-centered posets.
Recall that a subset $X$ of a forcing notion $\P$ is \emph{centered} if any finitely many elements of $X$ have a lower bound in $\P$.
A forcing notion $\P$ is called \emph{$\sigma$-centered} if $\P = {\bigcup}_{n \in \omega} {\P}_{n}$, where each ${\P}_{n}$ is centered.
\emph{Martin's axiom for $\sigma$-centered posets}, denoted \emph{$\MA(\sigma-\textrm{centered})$}, is the following statement: for every $\sigma$-centered poset $\P$ and every collection $\XX$ of fewer than $\c = {2}^{{\aleph}_{0}}$ many dense subsets of $\P$, there is a filter $G \subset \P$ such that $\forall D \in \XX\[G \cap D \neq 0\]$.
$\MA(\sigma-\textrm{centered})$ is a mild hypothesis; it is implied both by $\CH$ and by forcing axioms such as the Proper Forcing Axiom ($\PFA$).
It has some status as a ``quasi-axiom'' because it is a forcing axiom for a class of very well-behaved posets, and last but not least, it allows us to build P-points in a generic manner.
For these reasons it is arguable that $\MA(\sigma-\textrm{centered})$ is the right axiom under which to study the global structure of the P-points.

We should point out that $\MA(\sigma-\textrm{centered})$ is equivalent to the statement that $\p = \c$.
A family $F \subset \cube$ is said to have the \emph{finite intersection property (FIP)} if for any ${a}_{0}, \dotsc, {a}_{k} \in F$, ${a}_{0} \cap \dotsb \cap {a}_{k}$ is infinite.
$\p$  is the minimal cardinal $\kappa$ such that there is a family $F \subset \cube$ of size $\kappa$ with the FIP, but for which there is no $b \in \cube$ such that $\forall a \in F\[b \; {\subset}^{\ast} \; a\]$.

Among other results, Blass~\cite{blassrk} showed that $\MA(\sigma-\textrm{centered})$ implies that both ${\omega}_{1}$ and $\Ra$ (the real numbers ordered as usual) can be embedded into the P-points under the Rudin-Keisler ordering.
He posed the following question in his paper\footnote{Question 4 of \cite{blassrk} asks explicitly only about ordinals; but given the other results in that paper, the more general question is implicit.}:
\begin{Question}[Blass, 1973] \label{q:mainq}
 Assuming $\MA(\sigma-centered)$, what partial orders can be embedded into the P-points with respect to the Rudin-Keisler ordering?
\end{Question}
Some of Blass' results from \cite{blassrk} were reproved much later for the case of Tukey reducibility of ultrafilters.
The general notion of Tukey reducibility between directed quasi-orders arose with the Moore-Smith theory of convergence in topological spaces.
We say that a quasi-order $\langle D, \leq \rangle$ is \emph{directed} if any two members of $D$ have an upper bound in $D$.
A set $X \subset D$ is \emph{unbounded in $D$} if it doesn't have an upper bound in $D$.
A set $X \subset D$ is said to be \emph{cofinal in $D$} if $\forall y \in D\, \exists x \in X \left[y \leq x \right]$.
Given directed sets $D$ and $E$, a map $f: D \rightarrow E$ is called a \emph{Tukey map} if the image of every unbounded subset of $D$ is unbounded in $E$.
A map $g: E \rightarrow D$ is called a
\emph{convergent map} if the image of every cofinal subset of $E$ is cofinal in $D$.
It is not difficult to show that there is a Tukey map $f: D \rightarrow E$ if and only if there is a convergent $g: E \rightarrow D$.
\begin{Def} \label{def:tukeyred}
We say that $D$ is \emph{Tukey reducible} to $E$, and we write $D \; {\leq}_{T} \; E$ if there is a convergent map $g: E \rightarrow D$.
We say that $D$ and $E$ are \emph{Tukey equivalent} or have the same \emph{cofinal type} if both $D \; {\leq}_{T} \; E$ and $E \; {\leq}_{T} \; D$ hold.
\end{Def}
The topological significance of these notions is that if $ D \; {\leq}_{T} \; E$, then any $D$-net on a topological space contains an $E$-subnet. 

If $\U$ is any ultrafilter on $\omega$, then $\langle \U, \supset \rangle$ is a directed set.
When ultrafilters are viewed as directed sets in this way, Tukey reducibility is a coarser quasi order than RK reducibility.
In other words, if $\U \; {\leq}_{RK} \; \VV$, then $\U \; {\leq}_{T} \; \VV$.
In contrast with Kunen's theorem discussed above it is unknown whether it is possible to construct two ultrafilters on $\omega$ that not Tukey equivalent using only $\ZFC$.
For any $\XX, \YY \subset \Pset(\omega)$, a map $\phi: \XX \rightarrow \YY$ is said to be \emph{monotone} if $\forall a, b \in \XX\[a \subset b \implies \phi(a) \subset \phi(b)\]$, and $\phi$ is said to be \emph{cofinal in $\YY$} if $\forall b \in \YY \exists a \in \XX \[\phi(a) \subset b\]$.
It is a useful and easy fact that if $\U$ and $\VV$ are ultrafilters on $\omega$, then $\U \; {\leq}_{T} \; \VV$ iff there exists a $\phi: \VV \rightarrow \U$ that is monotone and cofinal in $\U$.

The order ${\leq}_{T}$ on the class of ultrafilters and particularly on the class of P-points has been studied recently in \cite{milovich}, \cite{tukey}, and \cite{dt}.
Dobrinen and Todorcevic~\cite{dt} showed that ${\omega}_{1}$ can be embedded into the P-points under the Tukey order, and Raghavan (unpublished) showed the same for $\Ra$.
These results rely on the fact, discovered by Dobrinen and Todorcevic~\cite{dt}, that if $\U$ and $\VV$ are P-points and $\U \; {\leq}_{T} \; \VV$, then there is always a continuous monotone map $\phi: \Pset(\omega) \rightarrow \Pset(\omega)$ such that $\phi \restrict \VV: \VV \rightarrow \U$ is cofinal in $\U$.
We will need a refinement of this fact for our construction in this paper.
This refinement will be proved in Lemma \ref{lem:localmap}.

These results of Dobrinen and Todorcevic~\cite{dt} and Raghavan rework Blass' arguments from \cite{blassrk} in the context of the Tukey ordering, and motivate us to ask the analogue of Question \ref{q:mainq} for this ordering also.
The main aim of this paper is to treat Question \ref{q:mainq} as well as its Tukey analogue.
We will prove the following theorem.
\begin{mth} \label{thm:mth}
 Assume $\MA(\sigma-\textrm{centered})$.
 Then there is a sequence of P-points $\langle {\U}_{\[a\]}: \[a\] \in \Pset(\omega) \slash \FIN \rangle$ such that
 \begin{enumerate}
  \item
  if $a \; {\subset}^{\ast} \; b$, then ${\U}_{\[a\]} \; {\leq}_{RK} \; {\U}_{\[b\]}$;
  \item
  if $b \; {\not\subset}^{\ast} \; a$, then ${\U}_{\[b\]} \; {\not\leq}_{T} \; {\U}_{\[a\]}$.
 \end{enumerate}
\end{mth}
Here $\FIN$ is the ideal of finite sets in the Boolean algebra $\Pset(\omega)$, and $\Pset(\omega) \slash \FIN$ is the quotient algebra.
For each $a \in \Pset(\omega)$, $\[a\]$ denotes the equivalence class of $a$ in $\Pset(\omega) \slash \FIN$.
Thus the theorem says that $\Pset(\omega) \slash \FIN$ with its natural partial order embeds into the class of P-points with respect to both Rudin-Keisler and Tukey reducibility.
It is well-known that every partial order of size at most $\c$ can be embedded into $\Pset(\omega) \slash \FIN$.
\begin{Cor} \label{cor:maincor}
 Under $\MA(\sigma-\textrm{centered})$ any partial order of size at most $\c$ embeds into the P-points both under RK and Tukey reducibility.
\end{Cor}
As far as we are aware, Corollary \ref{cor:maincor} is the first new piece of information on Question \ref{q:mainq} since Blass' work in \cite{blassrk}.
Since there are only $\c$ many functions from $\omega$ to $\omega$ and also only $\c$ many continuous functions from $\Pset(\omega)$ to $\Pset(\omega)$, any given P-point can have at most $\c$ many ultrafilters below it both with respect to RK and Tukey reducibility.
Therefore Corollary \ref{cor:maincor} is the best possible result for partial orders having a greatest element.
However it does not settle which partial orders of size greater than $\c$ can be embedded into the P-points (see Section \ref{sec:questions} for further discussion of what remains open).

Theorem \ref{thm:main} is proved using the technique of normed creatures pioneered by Shelah and his coauthors.
While this method is usually used for getting consistency results in set theory of the reals (see \cite{creaturebook}), it is a flexible method that can also be used for carrying out constructions from forcing axioms.
The method we develop in this paper for building ultrafilters is likely to be applicable to questions that ask whether certain classes of P-points can be distinguished from each other.
For instance, the questions posed at the end of \cite{equinumerosity} about interval P-points are likely to be amenable to our methods.
We can also modify the methods in this paper to shed a bit more light on Blass' original Question \ref{q:mainq}.
We have been able to prove the following theorem, which will be exposed in a future publication.
\begin{Theorem} \label{thm:futurepaper}
 Assume $\MA(\sigma-\textrm{centered})$.
 The ordinal ${\c}^{+}$ can be embedded into the P-points both under RK and Tukey reducibility.
\end{Theorem}

We end this introduction by fixing some notational conventions that will apply to the entire paper.
$A \subset B$ iff $\forall x\[x \in A \implies x \in B\]$, so the symbol ``$\subset$'' does not denote proper subset.
``$\forallbutfin x$ \ldots'' abbreviates the quantifier ``for all but finitely many $x$ \ldots'' and ``$\existsinf x$ \ldots'' stands for ``there exist infinitely many $x$ such that \dots''.
Given sets $X$ and $Y$, ${X}^{Y}$ denotes the collection of all functions from $Y$ to $X$.
Given a set $a$, $\Pset(a)$ denotes the power set of $a$.
$\cube$ refers to the collection of all infinite subsets of $\omega$, and ${\[\omega\]}^{< \omega}$ is the collection of all finite subsets of $\omega$.
A filter $\F$ on $\omega$ is required to be both \emph{proper}, meaning $0 \notin \F$, and \emph{non-principal}, meaning that $\forall F \in {\[\omega\]}^{< \omega}\[\omega \setminus F \in \F\]$.
Finally $A \; {\subset}^{\ast} \; B$ means $A \setminus B$ is finite and $A \; {=}^{\ast} \; B$ means $A \; {\subset}^{\ast} \; B$ and $B \; {\subset}^{\ast} \; A$.
\section{The construction} \label{sec:main}
We will build a set of ultrafilters $\{{\U}_{A}: A \in \XX\}$, where $\XX$ is some set of representatives for $\Pset(\omega)/\FIN$.
We will also build a corresponding set of maps in $\BS$, $\{{\pi}_{B, A}: A, B \in \XX \wedge A \; {\subset}^{\ast} \; B\}$, ensuring that if $A \; {\subset}^{\ast} \; B$ are any two members of $\XX$, then ${\pi}_{B, A}$ is an RK-map from ${\U}_{B}$ to ${\U}_{A}$.
We first define the notion of a creature needed for the construction and establish its most important properties.
\begin{Def} \label{def:creature}
 Let $A$ be a non-empty finite set.
 Say that $u$ is a \emph{creature acting on $A$} if $u$ is a pair of sequences $\langle \langle {u}_{a}: a \subset A \rangle, \langle {\pi}_{u, b, a}: a \subset b \subset A \rangle \rangle$ such that the following things hold:
 \begin{enumerate}
  \item
  each ${u}_{a}$ is a non-empty finite set;
  \item
  ${\pi}_{u, b, a}: {u}_{b} \rightarrow {u}_{a}$ is an onto function;
  \item
  if $a \subset b \subset c$, then ${\pi}_{u, c, a} = {\pi}_{u, b, a} \circ {\pi}_{u, c, b}$. 
 \end{enumerate}
 The collection of all creatures acting on $A$ is denoted $\cre(A)$.
 Strictly speaking of course $\cre(A)$ is a proper class, but we may restrict ourselves to the ones in $H(\omega)$. 
\end{Def}
The idea of this definition is that $u$ acts on the finite bit of information available to it to produce approximations to sets that will end up in various ultrafilters and also approximations to various RK maps.
More explicitly, if $X \in \Pset(\omega)$ and $A$ is some appropriately chosen finite set, then ${u}_{X \cap A}$ is an approximation to some set in the ultrafilter ${\U}_{X}$.
Similarly if $X \; {\subset}^{\ast} \; Y$ and if $X \cap A \subset Y \cap A$, then ${\pi}_{u, Y \cap A, X \cap A}$ approximates the RK map ${\pi}_{Y, X}$.
\begin{Def} \label{Def:sigma}
For a non-empty finite set $A$ and $u \in \cre(A)$, $\Sigma(u)$ denotes the collection of all $v \in \cre(A)$ such that:
\begin{enumerate}
 \item
 for each $a \subset A$, ${v}_{a} \subset {u}_{a}$;
 \item
 for each $a \subset b \subset A$, ${\pi}_{v, b, a} = {\pi}_{u, b, a} \restrict {v}_{b}$.
\end{enumerate}
\end{Def}
Note that if $v \in \Sigma(u)$, then $\Sigma(v) \subset \Sigma(u)$.
\begin{Def} \label{def:norm}
 For a non-empty finite set $A$, define the \emph{norm} of $u \in \cre(A)$, denoted $\nor(u)$, as follows.
 We first define by induction on $n \in \omega$, the relation $\nor(u) \geq n$ by the following clauses:
 \begin{enumerate}
  \item 
  $\nor(u) \geq 0$ always holds;
  \item
  $\nor(u) \geq n + 1$ iff 
  \begin{enumerate}
   \item[(a)]
   for each $a \subset A$, if ${u}_{a} = {u}^{0} \cup {u}^{1}$, then there exist $v \in \Sigma(u)$ and $i \in 2$ such that $\nor(v) \geq n$ and ${v}_{a} \subset {u}^{i}$; 
   \item[(b)]
   for any $a, b \subset A$, if $b \not\subset a$, then for every function $F: \Pset({u}_{a}) \rightarrow {u}_{b}$, there exists $v \in \Sigma(u)$ such that $\nor(v) \geq n$ and $F''\Pset({v}_{a}) \cap {v}_{b} = 0$.
  \end{enumerate}
 \end{enumerate}
 Define $\nor(u) = \max\{n \in \omega: \nor(u) \geq n\}$.
\end{Def}
It is easily seen that if $u \in \cre(A)$, $v \in \Sigma(u)$, and $\nor(v) \geq n$, then $\nor(u) \geq n$ as well.
It follows that for any $u \in \cre(A)$ if $\nor(u) \geq k$, then for all $n \leq k$, $\nor(u) \geq n$.
Because of the requirement that both $A$ and ${u}_{a}$ be non-empty, $\nor(u)$ is well-defined for every $u \in \cre(A)$.
To elaborate, if $k \in \omega$, $u \in \cre(A)$, and $\nor(u) \geq k + 1$, then since $0, A \subset A$, and $A \neq 0$, clause (2b) applies to $0$ and $A$.
By definition ${u}_{A} \neq 0$; fix ${x}_{0} \in {u}_{A}$.
Define a function $F: \Pset({u}_{0}) \rightarrow {u}_{A}$ by stipulating that $F(y) = {x}_{0}$, for every $y \in \Pset({u}_{0})$.
By (2b) there exists $v \in \Sigma(u)$ such that $\nor(v) \geq k$ and $F''\Pset({v}_{0}) \cap {v}_{A} = 0$.
Thus ${x}_{0} \notin {v}_{A}$ because ${x}_{0} \in F''\Pset({v}_{0})$.
As ${v}_{A} \neq 0$, we can choose ${x}_{1} \in {v}_{A}$.
Then ${x}_{0}, {x}_{1} \in {u}_{A}$ and ${x}_{1} \neq {x}_{0}$.
So we conclude that $\lc {u}_{A} \rc \geq 2$, if $\nor(u) \geq k + 1$.
Next, using this fact and clause (2a), a straightforward induction on $k \in \omega$ shows that for any $u \in \cre(A)$, if $\nor(u) \geq k$, then $\lc {u}_{A} \rc \geq k$.
This shows that $\nor(u)$ is well-defined.
Clause 2(a) ensures that we can construct ultrafilters, while clause 2(b) is needed to ensure that if $X, Y \in \XX$ and $Y \; {\not\subset}^{\ast}\; X$, then ${\U}_{Y} \; {\not\leq}_{T} \; {\U}_{X}$.  

The next lemma is a special case of a much more general theorem.
It is a Ramsey type theorem for a finite product of finite sets.
We only prove the special case which we use.
See \cite{creaturebook}, \cite{carlosstevo}, and \cite{jindraprodramsey} for far-reaching generalizations of this lemma.
\begin{Lemma} \label{lem:norm}
 For each $n < \omega$, for each $0 < l < \omega$, and for each $k < l$, there exists $0 < i(n, l, k) < \omega$ such that:
 \begin{enumerate}
  \item
  for each $n \in \omega$, $0 < l < \omega$, and $0 < m \leq l$, if $\langle {F}_{k}: k < m \rangle$ is a sequence of sets such that $\forall k < m\[\lc {F}_{k} \rc = i(n + 1, l, k)\]$ and if $\displaystyle\prod_{k < m}{F}_{k} = {X}_{0} \cup {X}_{1}$, then there exist $j \in 2$ and a sequence $\langle {E}_{k}: k < m \rangle$ such that $\forall k < m\[{E}_{k} \subset {F}_{k} \wedge \lc {E}_{k} \rc = i(n, l, k)\]$ and $\left(\displaystyle\prod_{k < m}{E}_{k} \right) \subset {X}_{j}$.
  \item
  for each $n < \omega$, each $0 < l < \omega$, and each $k < l$, $i(n + 1, l, k) \geq {2}^{x(n, l)} + i(n, l, k)$, where $x(n, l) = \displaystyle\prod_{k < l}i(n, l, k)$.
 \end{enumerate}
\end{Lemma}
\begin{proof}
 We define $i(n, l, k)$ by induction on $n \in \omega$ and for a fixed $n$ and a fixed $0 < l < \omega$, by induction on $k < l$.
 Put $i(0, l, k) = 1$ for all $0 < l < \omega$ and $k < l$.
 Fix $n \in \omega$.
 Suppose that $i(n, l, k)$ is given for all $0 < l < \omega$ and all $k < l$.
 Fix $0 < l < \omega$.
 We define $i(n + 1, l, k)$ by induction on $k < l$.
 Let $x(n, l)$ be as in (2) above.
 Note that $0 < x(n, l) < \omega$ and that for any $k < l$, $0 < i(n, l, k) < {2}^{x(n, l)} + i(n, l, k) < \omega$.
 Now fix $k < l$ and assume that $i(n + 1, l, k')$ has been defined for all $k' < k$.
 Define $y(n + 1, l, k) = \displaystyle\prod_{k' < k}i(n + 1, l, k')$ (when $k = 0$ this product is taken to be $1$) and let $z(n + 1, l, k) = {2}^{y(n + 1, l, k)}i(n, l, k)$.
 Note that $0 < z(n + 1, l, k) < \omega$.
 Put $i(n + 1, l, k) = \max\{z(n + 1, l, k), {2}^{x(n, l)} + i(n, l, k)\}$.
 Thus $0 < i(n + 1, l, k) < \omega$ and $i(n + 1, l, k) \geq {2}^{x(n, l)} + i(n, l, k)$ as needed for (2).
 
 To verify (1) fix $n \in \omega$ and $0 < l < \omega$.
 We induct on $0 < m \leq l$.
 Suppose $m = 1$ and suppose $ \lc {F}_{0} \rc = i(n + 1, l, 0)$ and suppose that ${F}_{0} = {X}_{0} \cup {X}_{1}$.
 Then $i(n + 1, l, 0) \geq 2i(n, l, 0)$.
 So there exists $j \in 2$ and ${E}_{0} \subset {X}_{j} \subset {F}_{0}$ such that $\lc {E}_{0} \rc = i(n, l, 0)$, as needed.
 
 Now fix $0 < m < m + 1 \leq l$ and suppose that the required statement holds for $m$.
 Let $\langle {F}_{k}: k < m + 1 \rangle$ be a sequence of sets such that $\forall k < m + 1\[\lc {F}_{k} \rc = i(n + 1, l, k)\]$ and suppose that $\displaystyle\prod_{k < m + 1}{F}_{k} = {X}_{0} \cup {X}_{1}$.
 Let $\langle {\sigma}_{i}: i < y(n + 1, l, m) \rangle$ enumerate the members of $\displaystyle\prod_{k < m}{F}_{k}$.
 Build a sequence $\langle {E}^{i}_{m}: -1 \leq i < y(n + 1, l, m)\rangle$ such that the following hold:
 \begin{enumerate}
  \item[(3)]
  ${E}^{-1}_{m} \subset {F}_{m}$ and $\forall -1 \leq i < i + 1 < y(n + 1, l, m)\[{E}^{i + 1}_{m} \subset {E}^{i}_{m}\]$;
  \item[(4)]
  $\forall -1 \leq i < y(n + 1, l, m)\[\lc {E}^{i}_{m} \rc = {2}^{y(n + 1, l, m) - i - 1}i(n, l, m)\]$;
  \item[(5)]
  $\forall 0 \leq i < y(n + 1, l, m) \exists {j}_{i} \in 2 \forall x \in {E}^{i}_{m}\[{\left( {\sigma}_{i} \right)}^{\frown}{\langle x \rangle} \in {X}_{{j}_{i}}\]$.
 \end{enumerate}
 The sequence is constructed by induction.
 To start choose ${E}^{-1}_{m} \subset {F}_{m}$ of size equal to ${2}^{y(n + 1, l, m)}i(n, l, m)$ (possible because $\lc {F}_{m} \rc = i(n + 1, l, m) \geq {2}^{y(n + 1, l, m)}i(n, l, m)$).
 Now suppose that $-1 \leq i < i + 1 < y(n + 1, l, m)$ and that ${E}^{i}_{m}$ is given.
 For each $j \in 2$ let ${Z}_{j} = \{x \in {E}^{i}_{m}: {\left( {\sigma}_{i + 1} \right)}^{\frown}{\langle x \rangle} \in {X}_{j}\}$.
 Then ${E}^{i}_{m} = {Z}_{0} \cup {Z}_{1}$ and so there exist ${E}^{i + 1}_{m} \subset {E}^{i}_{m}$ and ${j}_{i + 1} \in 2$ such that $\lc {E}^{i + 1}_{m} \rc = {2}^{y(n + 1, l, m) - i - 2}i(n, l, m)$ and ${E}^{i + 1}_{m} \subset {Z}_{{j}_{i + 1}}$.
 It is then clear that ${E}^{i + 1}_{m}$ and ${j}_{i + 1}$ satisfy (3)-(5).
 This completes the construction of the sequence $\langle {E}^{i}_{m}: -1 \leq i  < y(n + 1, l, m)\rangle$.
 For $j \in 2$ define ${Y}_{j} = \{{\sigma}_{i}: 0 \leq i < y(n + 1, l, m)\wedge {j}_{i} = j\}$.
 It is clear that $\displaystyle\prod_{k < m}{F}_{k} = {Y}_{0} \cup {Y}_{1}$.
 So by the inductive hypothesis, there exist $j \in 2$ and a sequence $\langle {E}_{k}: k  < m \rangle$ such that $\forall k < m \[{E}_{k} \subset {F}_{k} \wedge \lc {E}_{k} \rc = i(n, l, k)\]$ and $\left( \displaystyle\prod_{k < m}{E}_{k} \right) \subset {Y}_{j}$.
 Now put ${E}_{m} = {E}^{y(n + 1, l, m) - 1}_{m}$.
 The sequence $\langle {E}_{k}: k < m + 1\rangle$ and $j \in 2$ are as needed.
 This completes the verification of (1) and the proof of the lemma.
\end{proof}
We use Lemma \ref{lem:norm} to show that there exist creatures of arbitrarily high norm.
This is an essential step to defining a partial order out of any notion of a creature.
In our case each condition of the partial order is an approximation to the final collection of ultrafilters and RK-maps.
\begin{Cor} \label{cor:norm}
 Let $A$ be a non-empty finite set and $l = {2}^{\lc A \rc}$.
 Suppose $\langle {s}_{k}: k < l \rangle$ is an enumeration of all the subsets of $A$ such that if $k' < k$, then ${s}_{k} \not\subset {s}_{k'}$.
 For each $a \subset A$, let ${D}_{a}$ denote $\{k < l: {s}_{k} \subset a\}$.
 For each $n \in \omega$, if $\langle {F}_{k}: k < l \rangle$ is any sequence of sets such that $\forall k < l\[\lc {F}_{k} \rc = i(n, l, k)\]$, then $u = \langle \langle {u}_{a}: a \subset A \rangle, \langle {\pi}_{u, b, a}: a \subset b \subset A \rangle \rangle$, where ${u}_{a} = \displaystyle\prod\{{F}_{k}: k \in {D}_{a}\}$ and ${\pi}_{u, b, a}(s) = s \restrict {D}_{a}$, is a member of $\cre(A)$ and has norm at least $n$. 
\end{Cor}
\begin{proof}
 Since $i(n, l, k)$ is always at least $1$, $u$ as defined above is always a member of $\cre(A)$ with $\nor(u) \geq 0$ regardless of what $n$ is.
 So the claim holds for $n = 0$.
 We assume that the claim is true for some $n \in \omega$ and check it for $n + 1$.
 Indeed let $\langle {F}_{k}: k < l \rangle$ be any sequence of sets with $\lc {F}_{k} \rc = i(n + 1, l, k)$ and let $u$ be defined as above from $\langle {F}_{k}: k < l \rangle$.
 Suppose that $a \subset A$ and that ${u}_{a} = {u}^{0} \cup {u}^{1}$.
 Then $X = \displaystyle\prod\{{F}_{k}: k < l \} = {X}_{0} \cup {X}_{1}$, where ${X}_{j} = \{s \in X: s \restrict {D}_{a} \in {u}^{j}\}$.
 By (1) of Lemma \ref{lem:norm} applied with $m = l$, there exist a sequence $\langle {E}_{k}: k < l \rangle$ and a $j \in 2$ such that ${E}_{k} \subset {F}_{k}$, $ \lc {E}_{k} \rc = i(n, l, k)$, and $\displaystyle\prod\{{E}_{k}: k < l\} \subset {X}_{j}$.
 Now if $v$ is defined from the sequence $\langle {E}_{k}: k < l \rangle$ as above, then by the inductive hypothesis $v \in \cre(A)$ and $\nor(v) \geq n$.
 Moreover it is clear that $v \in \Sigma(u)$ and that ${v}_{a} \subset {u}^{j}$.
 So this checks clause 2(a) of Definition \ref{def:norm}.
 
 For clause 2(b), fix $a, b \subset A$ with $b \not\subset a$.
 Let $F: \Pset({u}_{a}) \rightarrow {u}_{b}$ be any function.
 For each $k < l$, let ${G}_{k} \subset {F}_{k}$ with $\lc {G}_{k} \rc = i(n, l, k)$.
 This is possible to do because by (2) of Lemma \ref{lem:norm}, $\forall k < l\[\lc {F}_{k} \rc  = i(n + 1, l, k) \geq {2}^{x(n, l)} + i(n, l, k) \geq i(n, l, k)\]$, where $x(n, l)$ is defined as there.
 Note that ${D}_{b} \setminus {D}_{a} \neq 0$.
 Fix ${k}_{0} \in {D}_{b} \setminus {D}_{a}$.
 Let $e = \displaystyle\prod\{{G}_{k}: k \in {D}_{a}\}$ and let $o = \displaystyle\prod\{{G}_{k}: k \in l\}$.
 Let $M = \{s({k}_{0}): s \in F''\Pset(e)\}$.
 Then $M \subset {F}_{{k}_{0}}$ and $\lc M \rc \leq \lc \Pset(o) \rc = {2}^{x(n, l)}$.
 There exists ${E}_{{k}_{0}} \subset {F}_{{k}_{0}}$ such that $\lc {E}_{{k}_{0}} \rc = i(n, l, {k}_{0})$ and ${E}_{{k}_{0}} \cap M = 0$ because $\lc {F}_{{k}_{0}} \rc \geq {2}^{x(n, l)} + i(n, l, {k}_{0})$.
 For all $k \in l \setminus \{{k}_{0}\}$, let ${E}_{k} = {G}_{k}$.
 Then $\langle {E}_{k}: k < l \rangle$ is a sequence of sets such that $\forall k < l\[{E}_{k} \subset {F}_{k} \wedge \lc {E}_{k} \rc = i(n, l, k)\]$.
 So by the inductive hypothesis if $v$ is defined as above from $\langle {E}_{k}: k < l \rangle$, then $v \in \cre(A)$ and $\nor(v) \geq n$.
 Moreover $v \in \Sigma(u)$.
 We check that $F''\Pset({v}_{a}) \cap {v}_{b} = 0$.
 Since ${k}_{0} \notin {D}_{a}$, $\forall k \in {D}_{a}\[{E}_{k} = {G}_{k}\]$.
 Therefore ${v}_{a} = e$.
 So if $s \in F''\Pset({v}_{a}) \cap {v}_{b}$, then $s({k}_{0}) \in M$.
 On the other hand by the definition of ${v}_{b}$, $s({k}_{0}) \in {E}_{{k}_{0}}$.
 Hence $M \cap {E}_{{k}_{0}} \neq 0$, contradicting the choice of ${E}_{{k}_{0}}$.
 Therefore $F''\Pset({v}_{a}) \cap {v}_{b} = 0$.
 This concludes the verification of clause 2(b) of Definition \ref{def:norm} and that proof that $\nor(u) \geq n + 1$.
\end{proof}
One of the main features of the final construction will be that creatures will be allowed to ``shift'' their scene of action.
In fact, we will want to perform this shifting operation infinitely often.
The following two lemmas ensure that the two main features of a creature $u$, namely $\nor(u)$ and $\Sigma(u)$, are preserved while shifting.
\begin{Def} \label{def:shift}
 Let $A$ and $B$ be non-empty finite sets and suppose $h: B \rightarrow A$ is an onto function.
 Let $u$ be a creature acting on $B$.
 Define $h\[u\] = v = \langle \langle {v}_{a}: a \subset A \rangle, \langle {\pi}_{v, {a}^{\ast}, a}: a \subset {a}^{\ast} \subset A \rangle \rangle$ by the following clauses:
 \begin{enumerate}
  \item
  for all $a \subset A$, ${v}_{a} = {u}_{{h}^{-1}(a)}$;
  \item
  for all $a \subset {a}^{\ast} \subset A$, ${\pi}_{v, {a}^{\ast}, a} = {\pi}_{u, {h}^{-1}({a}^{\ast}), {h}^{-1}(a)}$.
 \end{enumerate}
\end{Def}
\begin{Lemma} \label{lem:shift1}
 Let $A$, $B$, $h$, $u$, and $v = h\[u\]$ be as in Definition \ref{def:shift}.
 Then $v$ is a creature acting on $A$.
 Moreover, for any $w \in \Sigma(u)$, $h\[w\] \in \Sigma(v)$.
\end{Lemma}
\begin{proof}
 For any $a \subset A$, ${h}^{-1}(a) \subset B$, and so ${v}_{a} = {u}_{{h}^{-1}(a)}$ is a non-empty finite set.
 Similarly if $a \subset {a}^{\ast} \subset A$, then ${h}^{-1}(a) \subset {h}^{-1}({a}^{\ast}) \subset B$, and so ${\pi}_{v, {a}^{\ast}, a} = {\pi}_{u, {h}^{-1}({a}^{\ast}), {h}^{-1}(a)}$ is an onto map from ${v}_{{a}^{\ast}} = {u}_{{h}^{-1}({a}^{\ast})}$ to ${u}_{{h}^{-1}(a)} = {v}_{a}$.
 Thus $v$ is a creature acting on $A$.
 
 Next, suppose that $w \in \Sigma(u)$.
 By the above $h\[w\]$ is a creature acting on $A$.
 If $a \subset A$, then ${\left( h\[w\] \right)}_{a}= {w}_{{h}^{-1}(a)} \subset {u}_{{h}^{-1}(a)} = {v}_{a}$.
 Likewise, if $a \subset {a}^{\ast} \subset A$, then ${\pi}_{h\[w\], {a}^{\ast}, a} = {\pi}_{w, {h}^{-1}({a}^{\ast}), {h}^{-1}(a)} = {\pi}_{u, {h}^{-1}({a}^{\ast}), {h}^{-1}(a)} \restrict {w}_{{h}^{-1}({a}^{\ast})} = {\pi}_{v, {a}^{\ast}, a} \restrict {\left( h\[w\] \right)}_{{a}^{\ast}}$.
 Thus $h\[w\] \in \Sigma(v)$.
\end{proof}
\begin{Lemma} \label{lem:shift2}
Let $A$, $B$, $h$, $u$, and $v$ be as in Definition \ref{def:shift}.
For each $n \in \omega$, if $\nor(u) \geq n$, then $\nor(v) \geq n$.
\end{Lemma}
\begin{proof}
 The proof is by induction on $n$.
 For $n = 0$, by Lemma \ref{lem:shift1} $v$ is a creature acting on $A$ and so $\nor(v) \geq 0$.
 Assume that it holds for $n$ and suppose $\nor(u) \geq n + 1$.
 We first check clause 2(a) of Definition \ref{def:norm}.
 Let $a \subset A$ and suppose that ${v}_{a}  = {v}^{0} \cup {v}^{1}$.
 Then ${h}^{-1}(a) \subset B$ and ${v}_{a} = {u}_{{h}^{-1}(a)} = {v}^{0} \cup {v}^{1}$.
 So there exists $w \in \Sigma(u)$ with $\nor(w) \geq n$ and $i \in 2$ such that ${w}_{{h}^{-1}(a)} \subset {v}^{i}$.
 By Lemma \ref{lem:shift1} $h\[w\] \in \Sigma(v)$ and by the induction hypothesis $\nor(h\[w\]) \geq n$.
 Also ${\left(h\[w\]\right)}_{a} = {w}_{{h}^{-1}(a)} \subset {v}^{i}$.
 This checks clause 2(a) of Definition \ref{def:norm}.
 
 For clause 2(b), fix $a, {a}^{\ast} \subset A$ and suppose that ${a}^{\ast} \not\subset a$.
 Let $F: \Pset({v}_{a}) \rightarrow {v}_{{a}^{\ast}}$.
 We have ${h}^{-1}(a), {h}^{-1}({a}^{\ast}) \subset B$.
 Moreover ${a}^{\ast} \setminus a \neq \emptyset$.
 Since $h$ is onto, ${h}^{-1}({a}^{\ast} \setminus a) = {h}^{-1}({a}^{\ast}) \setminus {h}^{-1}(a) \neq \emptyset$.
 So ${h}^{-1}({a}^{\ast}) \not\subset {h}^{-1}(a)$.
 Also $F: \Pset({u}_{{h}^{-1}(a)}) \rightarrow {u}_{{h}^{-1}({a}^{\ast})}$.
 As $\nor(u) \geq n + 1$, we can find $w \in \Sigma(u)$ with $\nor(w) \geq n$ such that $F''\Pset({w}_{{h}^{-1}(a)}) \cap {w}_{{h}^{-1}({a}^{\ast})} = 0$.
 By Lemma \ref{lem:shift1} $h\[w\] \in \Sigma(v)$ and by the inductive hypothesis $\nor(h\[w\]) \geq n$.
 Also ${\left( h\[w\] \right)}_{a} = {w}_{{h}^{-1}(a)}$ and ${\left( h\[w\] \right)}_{{a}^{\ast}} = {w}_{{h}^{-1}({a}^{\ast})}$.
 Therefore $F''\Pset({\left( h\[w\] \right)}_{a}) \cap {\left( h\[w\] \right)}_{{a}^{\ast}} = 0$.
 This checks that $\nor(v) \geq n + 1$ and concludes the proof.
\end{proof}
We are now ready to define the forcing poset which we use.
We define a version of the poset that makes sense even in the absence of $\MA(\sigma-\textrm{centered})$, though $\MA(\sigma-\textrm{centered})$ is needed for the various density arguments. 
\begin{Def} \label{def:q}
 We say that $q$ is a \emph{standard sequence} if $q$ is a pair $\langle {I}_{q}, {U}_{q}\rangle$ such that:
 \begin{enumerate}
  \item
  ${I}_{q} = \langle {I}_{q, n}: n \in \omega \rangle$ is a sequence of non-empty finite subsets of $\omega$ such that $\forall n \in \omega \[\max({I}_{q, n}) < \min({I}_{q, n + 1})\]$;
  \item
  ${U}_{q} = \langle {u}^{q, n}: n \in \omega \rangle$ is a sequence such that for each $n \in \omega$, ${u}^{q, n}$ is a creature acting on ${I}_{q, n}$; if $a \subset b \subset {I}_{q, n}$, then ${\pi}_{{u}^{q, n}, b, a}$ will be denoted ${\pi}_{q, b, a}$;
  \item
  for each $n \in \omega$ and $a \subset {I}_{q, n}$, ${u}^{q, n}_{a} \subset \omega$;
  \item
  if $n < n + 1$, then $\nor({u}^{q, n}) < \nor({u}^{q, n + 1})$, and for all $a \subset {I}_{q, n}$ and all $b \subset {I}_{q, n + 1}$, $\max({u}^{q, n}_{a}) < \min\left({u}^{q, n + 1}_{b}\right)$.
 \end{enumerate}
 $\Q$ denotes the set of all standard sequences.
\end{Def}
There are several natural partial orderings that can be defined on $\Q$.
However, we will not be using any ordering on $\Q$ in our construction.
\begin{Def} \label{def:0condition}
 $p$ is called a \emph{$0$-condition} if $p = \langle {\A}_{p}, {\C}_{p}, {\D}_{p}\rangle$ where:
 \begin{enumerate}
  \item
  ${\A}_{p} \subset \Pset(\omega)$, $0, \omega \in {\A}_{p}$, $\lc {\A}_{p} \rc < \c$, and $\forall A, B \in {\A}_{p}\[A \neq B \implies A  \; {\neq}^{\ast} \; B\]$;
  \item
  ${\D}_{p} = \langle {\DD}_{p, A}: A \in {\A}_{p} \rangle$ is a sequence of non-principal filters on $\omega$ with the property that for each $A \in {\A}_{p}$ there exists a family ${\F}_{p, A} \subset {\DD}_{p, A}$ with $\lc {\F}_{p, A} \rc < \c$ such that $\forall X \in {\DD}_{p, A}\exists Y \in {\F}_{p, A}\[Y \subset X\]$;
  \item
  ${\C}_{p} = \langle {\pi}_{p, B, A}: A, B \in {\A}_{p} \wedge A \; {\subset}^{\ast} \; B \rangle$ is a sequence of elements of $\BS$;
  \item
  for all $A, B \in {\A}_{p}$, if $A \; {\subset}^{\ast} \; B$, then $\forall X \in {\DD}_{p, B} \[{\pi}_{p, B, A}''X \in {\DD}_{p, A}\]$.
 \end{enumerate}
 ${\P}_{0} = \{p: p \ \text{is a} \ 0\text{-condition}\}$.
 Define an ordering on ${\P}_{0}$ as follows.
 For any ${p}_{0}, {p}_{1} \in {\P}_{0}$, ${p}_{1} \leq {p}_{0}$ iff ${\A}_{{p}_{1}} \supset {\A}_{{p}_{0}}$, $\forall A, B \in {\A}_{{p}_{0}} \[ A \; {\subset}^{\ast} \; B \implies {\pi}_{{p}_{1}, B, A} = {\pi}_{{p}_{0}, B, A}\]$, and $\forall A \in {\A}_{{p}_{0}}\[{\DD}_{{p}_{1}, A} \supset {\DD}_{{p}_{0}, A}\]$. 
\end{Def}
\begin{Def} \label{def:induce}
 Let $p \in {\P}_{0}$ and $q \in \Q$.
 We say that $q$ \emph{induces} $p$ if the following hold:
 \begin{enumerate}
  \item 
  Let $\BB$ denote the Boolean subalgebra of $\Pset(\omega)$ generated by ${\A}_{p}$; then for every infinite member $A$ of $\BB$, $\forallbutfin n \in \omega \[\lc A \cap {I}_{q, n} \rc < \lc A \cap {I}_{q, n + 1} \rc\]$;
  \item
  for each $A \in {\A}_{p}$ and each $X \in {\DD}_{p, A}$, $\forallbutfin n \in \omega \[{u}^{q, n}_{A \cap {I}_{q, n}} \subset X\]$;
  \item
  for each $A, B \in {\A}_{p}$ with $A \; {\subset}^{\ast} \; B$ the following holds:
  \begin{align*}
  \forallbutfin n \in \omega \[{\pi}_{p, B, A} \restrict {u}^{q, n}_{B \cap {I}_{q, n}} = {\pi}_{q, B \cap {I}_{q, n}, A \cap {I}_{q, n}}\].
  \end{align*}
 \end{enumerate}
\end{Def}
Note that if $p, p' \in {\P}_{0}$, $p \leq p'$, $q \in \Q$, and $q$ induces $p$, then $q$ also induces $p'$.
\begin{Lemma} \label{lem:ppoint1}
 Let $p \in {\P}_{0}$ and suppose $q \in \Q$ induces $p$.
 Define ${p}_{0} = \langle {\A}_{{p}_{0}}, {\C}_{{p}_{0}}, {\D}_{{p}_{0}}\rangle$, where ${\A}_{{p}_{0}} = {\A}_{p}$, ${\D}_{{p}_{0}} = \langle {\DD}_{{p}_{0}, A}: A \in {\A}_{{p}_{0}} \rangle$, where 
 \begin{align*}
 {\DD}_{{p}_{0}, A} = \left\{a \subset \omega: \left( {\bigcup}_{n \in \omega}{{u}^{q, n}_{A \cap {I}_{q, n}}}\right) \; {\subset}^{\ast} \; a \right\},
 \end{align*}
 and ${\C}_{{p}_{0}} = \langle {\pi}_{{p}_{0}, B, A}: A, B \in {\A}_{{p}_{0}} \wedge A \; {\subset}^{\ast} \; B \rangle$, where ${\pi}_{{p}_{0}, B, A} = {\pi}_{p, B, A}$.
 Then ${p}_{0} \in {\P}_{0}$, ${p}_{0} \leq p$, and $q$ induces ${p}_{0}$.
\end{Lemma}
\begin{proof}
 The only clause in Definition \ref{def:0condition} that is not obvious is (4).
 To check it fix $A, B \in {\A}_{{p}_{0}}$ with $A \; {\subset}^{\ast}\; B$.
 Fix $X \in {\DD}_{{p}_{0}, B}$.
 Since ${\pi}_{{p}_{0}, B, A} = {\pi}_{p, B, A}$, we would like to see that ${\pi}_{p, B, A}''X \in {\DD}_{{p}_{0}, A}$.
 By the definition of ${\DD}_{{p}_{0}, B}$, $ \left( {\bigcup}_{n \in \omega}{{u}^{q, n}_{B \cap {I}_{q, n}}} \right) \; {\subset}^{\ast}\; X$.
 Because of this and because $q$ induces $p$ and $A \; {\subset}^{\ast} \; B$, the following things hold:
 \begin{enumerate}
  \item
  $\forallbutfin n \in \omega \[{u}^{q, n}_{B \cap {I}_{q, n}} \subset X\]$;
  \item
  $\forallbutfin n \in \omega \[A \cap {I}_{q, n} \subset B \cap {I}_{q, n}\]$;
  \item
  $\forallbutfin n \in \omega \[{\pi}_{p, B, A} \restrict {u}^{q, n}_{B \cap {I}_{q, n}} = {\pi}_{q, B \cap {I}_{q, n}, A \cap {I}_{q, n}}\]$.
 \end{enumerate}
 Let $n \in \omega$ be arbitrary such that (1)-(3) hold.
 Then 
 \begin{align*}
 {\pi}_{q, B \cap {I}_{q, n}, A \cap {I}_{q, n}}: {u}^{q, n}_{B \cap {I}_{q, n}} \rightarrow {u}^{q, n}_{A \cap {I}_{q, n}} 
 \end{align*}
 is an onto function.
 So if $k \in {u}^{q, n}_{A \cap {I}_{q, n}}$, then for some $l \in {u}^{q, n}_{B \cap {I}_{q, n}} \subset X$ we have
 \begin{align*}
  k = {\pi}_{q, B \cap {I}_{q, n}, A \cap {I}_{q, n}}(l) = \left( {\pi}_{p, B, A} \restrict {u}^{q, n}_{B \cap {I}_{q, n}} \right)(l) = {\pi}_{p, B, A}(l)
 \end{align*}
 Therefore $k \in {\pi}_{p, B, A}''X$, and so ${u}^{q, n}_{A \cap {I}_{q, n}} \subset {\pi}_{p, B, A}''X$.
 Thus we have shown that $\forallbutfin n \in \omega \[{u}^{q, n}_{A \cap {I}_{q, n}} \subset {\pi}_{p, B, A}''X\]$, which implies ${\pi}_{p, B, A}''X \in {\DD}_{{p}_{0}, A}$.
 
 Checking that ${p}_{0} \leq p$ and that $q$ induces ${p}_{0}$ is straightforward.
\end{proof}
\begin{Def} \label{def:1condition}
 We say that a $0$-condition $p$ is \emph{finitary} if $\lc {\A}_{p} \rc < \omega$ and $\forall A \in {\A}_{p} \exists {\F}_{p, A} \subset {\DD}_{p, A}\[\lc {\F}_{p, A}\rc \leq \omega \wedge \forall X \in {\DD}_{p, A} \exists Y \in {\F}_{p, A}\[Y \subset X\]\]$.
 A $0$-condition $p$ is called a \emph{$1$-condition} if every finitary $p' \in {\P}_{0}$ that satisfies $p \leq p'$ is induced by some $q \in \Q$.
 Let ${\P}_{1} = \{p \in {\P}_{0}: p \ \text{is a} \ 1\text{-condition}\}$.
 We partially order ${\P}_{1}$ by the same ordering $\leq$ as ${\P}_{0}$.
\end{Def}
\begin{Lemma} \label{lem:nonempty}
 ${\P}_{1}$ is non-empty.
\end{Lemma}
\begin{proof}
 Let ${\A}_{p} = \{0, \omega\}$.
 Define ${i}_{0} = 0$ and ${i}_{n + 1} = {2}^{n + 1}$ for all $n \in \omega$.
 Let ${I}_{n} = [{i}_{n}, {i}_{n + 1})$ and find a sequence $U = \langle {u}^{n}: n \in \omega \rangle$ satisfying clauses (2)-(4) of Definition \ref{def:q} with respect to $I = \langle {I}_{n}: n \in \omega \rangle$ using Corollary \ref{cor:norm}.
 Then $q =\langle I, U \rangle \in \Q$.
 Let ${A}_{0} = {\bigcup}_{n \in \omega}{{u}^{n}_{0}}$ and let ${A}_{\omega} = {\bigcup}_{n \in \omega}{{u}^{n}_{{I}_{n}}}$.
 Both of these sets are infinite subsets of $\omega$.
 Let ${\DD}_{p, 0} = \{a \subset \omega: {A}_{0} \; {\subset}^{\ast} \; a \}$ and ${\DD}_{p, \omega} = \{a \subset \omega: {A}_{\omega} \; {\subset}^{\ast} \; a \}$.
 Let ${\D}_{p} = \langle {\DD}_{p, A}: A = 0 \vee A = \omega \rangle$.
 Define ${\pi}_{p, \omega, 0}, {\pi}_{p, 0, 0}, {\pi}_{p, \omega, \omega} \in \BS$ as follows.
 Fix $k \in \omega$.
 If $k \in {A}_{\omega}$, then ${\pi}_{p, \omega, 0}(k) = {\pi}_{{u}^{n}, {I}_{n}, 0}(k)$ and ${\pi}_{p, \omega, \omega}(k) = {\pi}_{{u}^{n}, {I}_{n}, {I}_{n}}(k)$, where $n$ is the unique member of $\omega$ such that $k \in {u}^{n}_{{I}_{n}}$; if $k \notin {A}_{\omega}$, then ${\pi}_{p, \omega, 0}(k) = 0 = {\pi}_{p, \omega, \omega}(k)$; if $k \in {A}_{0}$, then let ${\pi}_{p, 0, 0}(k) = {\pi}_{{u}^{n}, 0, 0}(k)$, where $n$ is the unique member of $\omega$ such that $k \in {u}^{n}_{0}$; if $k \notin {A}_{0}$, then put ${\pi}_{p, 0, 0}(k) = 0$.
 Let ${\C}_{p} = \langle {\pi}_{p, B, A}: A, B \in {\A}_{p} \wedge A \; {\subset}^{\ast} \; B \rangle$.
 Let $p = \langle {\A}_{p}, {\C}_{p}, {\D}_{p} \rangle$.
 It is easy to check that $p \in {\P}_{0}$ and that $q$ induces $p$.
 So $q$ also induces any $p' \in {\P}_{0}$ with $p \leq p'$.
 Thus $p \in {\P}_{1}$.
\end{proof}
${\P}_{1}$ is the poset that will be used in the construction.
As mentioned earlier, $\MA(\sigma-\textrm{centered})$ is not needed for the definition of ${\P}_{1}$ or to prove that it is non-empty, although it will be needed to prove most of its properties.
The first of these properties, proved in the next lemma, shows that there is a single standard sequence that induces the entire condition.
\begin{Lemma} [Representation Lemma]\label{lem:rep}
 Assume $\MA(\sigma-\textrm{centered})$.
 Every $p \in {\P}_{1}$ is induced by some $q \in \Q$.
\end{Lemma}
\begin{proof}
Fix $p \in {\P}_{1}$.
For each $A \in {\A}_{p}$ choose ${\F}_{p, A} \subset {\DD}_{p, A}$ as in (2) of Definition \ref{def:0condition}.
Define a partial order $\Ra$ as follows.
A condition $r \in \Ra$ iff $r = \langle {f}_{r}, {g}_{r}, {F}_{r}, {\Phi}_{r}\rangle$ where:
\begin{enumerate}
 \item
 $\langle {f}_{r}, {g}_{r} \rangle$ is an initial segment of some standard sequence --  that is, there exist ${n}_{r} \in \omega$ and a standard sequence $\langle I, U \rangle$ such that ${f}_{r} = I \restrict {n}_{r}$ and ${g}_{r} = U \restrict {n}_{r}$;
 \item
 ${F}_{r}$ is a finite subset of ${\A}_{p}$;
 \item
 ${\Phi}_{r}$ is a function with domain ${F}_{r}$ such that $\forall A \in {F}_{r}\[{\Phi}_{r}(A) \in {\DD}_{p, A}\]$.
\end{enumerate}
Partially order $\Ra$ by stipulating that $s \preceq r$ iff
\begin{enumerate}
 \item[(4)]
 ${f}_{s} \supset {f}_{r}$, ${g}_{s} \supset {g}_{r}$, ${F}_{s} \supset {F}_{r}$, and $\forall A \in {F}_{r}\[{\Phi}_{s}(A) \subset {\Phi}_{r}(A)\]$;
 \item[(5)]
 if ${\BB}_{r}$ is the Boolean subalgebra of $\Pset(\omega)$ generated by ${F}_{r}$, then for every $B \in {\BB}_{r}$, $\forall n \in {n}_{s} \setminus {n}_{r}\[{f}_{s}(n) \cap B \neq 0 \ \text{iff} \ B \ \text{is infinite}\]$;
 \item[(6)]
 for every infinite $B \in {\BB}_{r}$, 
 \begin{align*}
 \forall n \in {n}_{s}\[n + 1 \in {n}_{s} \setminus {n}_{r} \implies \lc B \cap {f}_{s}(n) \rc < \lc B \cap {f}_{s}(n + 1)\rc\];
 \end{align*}
 \item[(7)]
 for each $A \in {F}_{r}$, $\forall n \in {n}_{s} \setminus {n}_{r} \[{\left( {g}_{s}(n) \right)}_{\left(A \cap {f}_{s}(n)\right)} \; \subset \; {\Phi}_{r}(A)\]$;
 \item[(8)]
 for each $A, B \in {F}_{r}$, if $A \; {\subset}^{\ast} \; B$, then 
 \begin{align*}
 \forall n \in {n}_{s} \setminus {n}_{r}\[{\pi}_{p, B, A} \restrict \left( {\left( {g}_{s}(n) \right)}_{B \cap {f}_{s}(n)}\right) = {\pi}_{{g}_{s}(n), B \cap {f}_{s}(n), A \cap {f}_{s}(n)}\].
 \end{align*}
\end{enumerate}
It is easily checked that $\langle \Ra, \preceq \rangle$ is a $\sigma$-centered poset.
It is also easy to check that for each $A \in {\A}_{p}$ and each $Y \in {\F}_{p, A}$, ${R}_{A, Y} = \{s \in \Ra: A \in {F}_{s} \wedge {\Phi}_{s}(A) \subset Y \}$ is dense in $\Ra$.
Now check the following claim.
\begin{Claim} \label{claim:rep1}
 For each $n \in \omega$, ${R}_{n} = \{s \in \Ra: n < {n}_{s}\}$ is dense in $\Ra$. 
\end{Claim}
\begin{proof}
The proof is by induction on $n$.
Fix $n$ and suppose the claim is true for all $m < n$.
Let $r \in \Ra$.
By the inductive hypothesis, we may assume that $n \subset {n}_{r}$.
If $n < {n}_{r}$, then there is nothing to do, so we assume $n = {n}_{r}$ and define $s$ so that ${n}_{s} = n + 1$. 
Also $0, \omega \in {\A}_{p}$.
So we may assume that $\{0, \omega\} \subset {F}_{r}$.
Let ${\BB}_{r}$ be the Boolean subalgebra of $\Pset(\omega)$ generated by ${F}_{r}$.
This is finite.
So we can find a finite, non-empty set ${f}_{s}(n) \subset \omega$ such that:
\begin{enumerate}
 \item[(9)]
 for any finite $B \in {\BB}_{r}$, $B \cap {f}_{s}(n) = 0$;
 \item[(10)]
 for any infinite $B \in {\BB}_{r}$, $B \cap {f}_{s}(n) \neq 0$;
 \item[(11)]
 if $n > 0$, then $\min({f}_{s}(n)) > \max({f}_{r}(n - 1))$ and for any infinite $B \in {\BB}_{r}$, $\lc {f}_{r}(n - 1) \cap B \rc < \lc {f}_{s}(n) \cap B \rc$.
\end{enumerate}
Now we will define a finitary ${p}_{0} \in {\P}_{0}$ with $p \leq {p}_{0}$.
Let ${\A}_{{p}_{0}} = {F}_{r}$.
We define by induction on $n \in \omega$ sequences ${\bar{X}}_{n} = \langle {X}_{A, n}: A \in {\A}_{{p}_{0}}\rangle$ such that $\forall n \in \omega \forall A \in {\A}_{{p}_{0}}\[{X}_{A, n} \in {\DD}_{p, A} \wedge {X}_{A, n + 1} \subset {X}_{A, n}\]$.
Define ${X}_{A, 0} = {\Phi}_{r}(A)$, for all $A \in {\A}_{{p}_{0}}$.
Suppose that ${\bar{X}}_{n}$ having the required properties is given for some $n \in \omega$.
For each $A \in {\A}_{{p}_{0}}$, define ${X}_{A, n + 1} = {X}_{A, n} \cap \left( \bigcap \left\{ {\pi}_{p, B, A}'' {X}_{B, n}: B \in {\A}_{{p}_{0}} \wedge A \; {\subset}^{\ast} \; B \right\} \right)$.
It is easy to see that ${\bar{X}}_{n + 1}$ has the required properties.
This completes the definition of the ${\bar{X}}_{n}$.
Now define ${\DD}_{{p}_{0}, A} = \{a \subset \omega: \exists n \in \omega \[{X}_{A, n} \; {\subset}^{\ast} \; a \]\}$, for each $A \in {\A}_{{p}_{0}}$.
Note $\forall A \in {\A}_{{p}_{0}}\forall n \in \omega\[{X}_{A, n} \in {\DD}_{{p}_{0}, A}\]$.
Let ${\D}_{{p}_{0}} = \langle {\DD}_{{p}_{0}, A}: A \in {\A}_{{p}_{0}}\rangle$.
Finally, for any $A, B \in {\A}_{{p}_{0}}$ with $A \; {\subset}^{\ast} \; B$, let ${\pi}_{{p}_{0}, B, A} = {\pi}_{p, B, A}$ and let ${\C}_{{p}_{0}} = \langle {\pi}_{{p}_{0}, B, A}: B, A \in {\A}_{{p}_{0}} \wedge A \; {\subset}^{\ast} \; B \rangle$.
Then ${p}_{0} = \langle {\A}_{{p}_{0}}, {\C}_{{p}_{0}} , {\D}_{{p}_{0}} \rangle$ is in ${\P}_{0}$, $p \leq {p}_{0}$, and ${p}_{0}$ is finitary.
So by hypothesis we can fix ${q}_{0} \in \Q$ inducing ${p}_{0}$.
Since ${\A}_{{p}_{0}}$ and ${\BB}_{r}$ are both finite, it is possible to find $m \in \omega$ such that:
\begin{enumerate}
 \item[(12)]
 for each $A \in {\BB}_{r}$, ${I}_{{q}_{0}, m} \cap A \neq 0$ iff $A$ is infinite; moreover for every infinite $A \in {\BB}_{r}$, $\lc A \cap {I}_{{q}_{0}, m} \rc \geq \lc A \cap {f}_{s}(n) \rc$;
 \item[(13)]
 for each $A \in {\A}_{{p}_{0}}$, ${u}^{{q}_{0}, m}_{A \cap {I}_{{q}_{0}, m}} \subset {X}_{A, 0}$;
 \item[(14)]
 for all $A, B \in {\A}_{{p}_{0}}$ with $A \; {\subset}^{\ast} \; B$, ${\pi}_{{p}_{0}, B, A} \restrict {u}^{{q}_{0}, m}_{B \cap {I}_{{q}_{0},m}} = {\pi}_{{q}_{0}, B \cap {I}_{{q}_{0}, m}, A \cap {I}_{{q}_{0}, m}}$;
 \item[(15)]
 if $n > 0$, then $\nor({u}^{{q}_{0}, m}) > \nor({g}_{r}(n - 1))$ and for every $a \subset {f}_{r}(n - 1)$ and every $b \subset {I}_{{q}_{0}, m}$, $\min({u}^{{q}_{0}, m}_{b}) > \max({\left( {g}_{r}(n - 1) \right)}_{a})$.
\end{enumerate}
Let $\{{A}_{0}, \dotsc, {A}_{l}\}$ enumerate the members of ${\A}_{{p}_{0}}$.
For each $\sigma \in {2}^{l + 1}$ define ${b}_{\sigma} = \left( \bigcap \left\{{A}_{i}: \sigma(i) = 0 \right\} \right) \cap \left( \bigcap \left\{ \omega \setminus {A}_{i}: \sigma(i) = 1 \right\}\right)$ (in this definition $\bigcap 0 = \omega$).
Let $T = \{\sigma \in {2}^{l + 1}: {b}_{\sigma} \ \text{is infinite}\}$.
Because of (9), (10), and (12), ${f}_{s}(n) = {\bigcup}_{\sigma \in T} \left( {b}_{\sigma} \cap {f}_{s}(n)\right)$ and ${I}_{{q}_{0}, m} = {\bigcup}_{\sigma \in T}\left( {b}_{\sigma} \cap {I}_{{q}_{0}, m}\right)$.
Also if $\sigma \neq \tau$, then ${b}_{\sigma} \cap {b}_{\tau} = 0$ and if $\sigma \in T$, then $\lc {b}_{\sigma} \cap {I}_{{q}_{0}, m}\rc \geq \lc {b}_{\sigma} \cap {f}_{s} (n) \rc \neq 0$.
Therefore there is an onto map $h: {I}_{{q}_{0}, m} \rightarrow {f}_{s}(n)$ such that $\forall \sigma \in T\[{h}^{-1}({b}_{\sigma} \cap {f}_{s}(n)) = {b}_{\sigma} \cap {I}_{{q}_{0}, m}\]$.
Let ${g}_{s}(n) = h\[{u}^{{q}_{0}, m}\]$.
Then by Lemmas \ref{lem:shift1} and \ref{lem:shift2}, ${g}_{s}(n)$ is a creature acting on ${f}_{s}(n)$, and if $n > 0$, then $\nor({g}_{s}(n)) > \nor({g}_{r}(n - 1))$.
Also if $a \subset {f}_{s}(n)$, then ${\left( {g}_{s} (n)\right)}_{a} = {u}^{{q}_{0}, m}_{{h}^{-1}(a)} \subset \omega$, and if $n > 0$, then for any $x \subset {f}_{r}(n - 1)$, $\max({\left( {g}_{r}(n - 1)\right)}_{x}) < \min({\left( {g}_{s}(n) \right)}_{a})$.
So if we define ${n}_{s} = n + 1 $, ${f}_{s} = {{f}_{r}}^{\frown}{\langle {f}_{s}(n) \rangle}$, ${g}_{s} = {{g}_{r}}^{\frown}{\langle {g}_{s}(n) \rangle}$, ${F}_{s} = {F}_{r}$, and ${\Phi}_{r} = {\Phi}_{s}$, then $s \in \Ra$.
We check that $s \preceq r$.
Clause (4) is obvious and clause (5) follows from (9) and (10).
Since ${n}_{s} \setminus {n}_{r} = \{n\}$, clause (6) just amounts to the second part of clause (11).

In order to check (7) and (8), we first make a preliminary observation.
For each $0 \leq i \leq l$, put ${T}_{i} = \{\sigma \in T: \sigma(i) = 0\}$.
Because of (9), (10), and (12) ${A}_{i} \cap {f}_{s}(n) = \bigcup \{{b}_{\sigma} \cap {f}_{s}(n): \sigma \in {T}_{i}\}$ and ${A}_{i} \cap {I}_{{q}_{0}, m} = \bigcup \{{b}_{\sigma} \cap {I}_{{q}_{0}, m}: \sigma \in {T}_{i}\}$.
Therefore for any $0 \leq i \leq l$, ${h}^{-1}({A}_{i} \cap {f}_{s}(n)) = \bigcup \{{h}^{-1}({b}_{\sigma} \cap {f}_{s}(n)): \sigma \in {T}_{i}\} = \bigcup \{{b}_{\sigma} \cap {I}_{{q}_{0}, m}: \sigma \in {T}_{i}\} = {A}_{i} \cap {I}_{{q}_{0}, m}$.
With this observation in hand, let us check (7) and (8).
Take any $A \in {F}_{r} = {\A}_{{p}_{0}}$.
There is $0 \leq i \leq l$ such that $A = {A}_{i}$ and ${\left( {g}_{s}(n) \right)}_{\left( {A}_{i} \cap {f}_{s}(n)\right)} = {u}^{{q}_{0}, m}_{{h}^{-1}({A}_{i} \cap {f}_{s}(n))} = {u}^{{q}_{0}, m}_{{A}_{i} \cap {I}_{{q}_{0}, m}} \subset {X}_{{A}_{i}, 0} = {\Phi}_{r}({A}_{i})$, as needed for (7).
For (8), fix $A, B \in {F}_{r} = {\A}_{{p}_{0}}$ with $A \; {\subset}^{\ast} \; B$.
Then for some $0 \leq i, j \leq l$, $A = {A}_{i}$ and $B = {A}_{j}$.
Observe that ${A}_{i} \setminus {A}_{j}$ is a finite member of ${\BB}_{r}$ because ${A}_{i} \; {\subset}^{\ast} \; {A}_{j}$.
Therefore by (9) $\left( {A}_{i} \setminus {A}_{j} \right) \cap {f}_{s}(n) = 0$, and ${A}_{i} \cap {f}_{s}(n) \subset {A}_{j} \cap {f}_{s}(n)$.
Therefore ${\pi}_{{g}_{s}(n), {A}_{j} \cap {f}_{s}(n), {A}_{i} \cap {f}_{s}(n)}$ is defined as is equal to ${\pi}_{{u}^{{q}_{0}, m}, {h}^{-1}({A}_{j} \cap {f}_{s}(n)), {h}^{-1}({A}_{i} \cap {f}_{s}(n))}$. 
So ${\pi}_{p, {A}_{j}, {A}_{i}} \restrict \left( {\left( {g}_{s}(n) \right)}_{{A}_{j} \cap {f}_{s}(n)} \right) = {\pi}_{{p}_{0}, {A}_{j}, {A}_{i}} \restrict \left( {u}^{{q}_{0}, m}_{{h}^{-1}({A}_{j} \cap {f}_{s}(n))}\right) = {\pi}_{{p}_{0}, {A}_{j}, {A}_{i}} \restrict \left( {u}^{{q}_{0}, m}_{{A}_{j} \cap {I}_{{q}_{0}, m}}\right) = {\pi}_{{u}^{{q}_{0}, m}, {A}_{j} \cap {I}_{{q}_{0}, m}, {A}_{i} \cap {I}_{{q}_{0}, m}} = {\pi}_{{u}^{{q}_{0}, m}, {h}^{-1}({A}_{j} \cap {f}_{s}(n)), {h}^{-1}({A}_{i} \cap {f}_{s}(n))} = {\pi}_{{g}_{s}(n), {A}_{j} \cap {f}_{s}(n), {A}_{i} \cap {f}_{s}(n)}$, which is exactly what is needed.
This checks $s \preceq r$ and completes the proof of the claim.
\end{proof}
Using $\MA(\sigma-\textrm{centered})$ we can find a filter $G \subset \Ra$ that meets every member of $\{{R}_{A, Y}: A \in {\A}_{p} \wedge Y \in {\F}_{p, A}\} \cup \{{R}_{n}: n \in \omega\}$ (recall that $\c$ is regular under $\MA(\sigma-\textrm{centered})$).
Let $I = \bigcup\{{f}_{r}: r \in G\}$ and $U = \bigcup\{{g}_{r}: r \in G\}$.
Then it is clear that $q = \langle I, U \rangle \in \Q$.
We check that $q$ induces $p$.
Let $\BB$ be the Boolean subalgebra of $\Pset(\omega)$ generated by ${\A}_{p}$.
Take $A \in \BB$.
Then there exist ${A}_{0}, \dotsc, {A}_{l} \in {\A}_{p}$ such that $A \in {\BB}_{0}$, where ${\BB}_{0}$ is the Boolean subalgebra of $\Pset(\omega)$ generated by $\{{A}_{0}, \dotsc, {A}_{l}\}$.
For each $0 \leq i \leq l$, ${\F}_{p, {A}_{i}}$ is non-empty.
Choosing ${Y}_{i} \in {\F}_{p, {A}_{i}}$, ${R}_{{A}_{i}, {Y}_{i}}$ is a dense open set met by $G$.
So there is $r \in G \cap \left( {\bigcap}_{i \leq l} {R}_{{A}_{i}, {Y}_{i}}\right)$.
Then $A \in {\BB}_{r}$.
For any $n \geq {n}_{r}$ there is $t \in G$ such that $t \preceq r$ and $n + 1 < {n}_{t}$.
Then if $A$ is infinite, then since $n + 1 \in {n}_{t} \setminus {n}_{r}$, by (6), we have $\lc A \cap {I}_{n} \rc = \lc  A \cap {f}_{t}(n) \rc < \lc A \cap {f}_{t}(n + 1)\rc = \lc A \cap {I}_{n + 1}\rc$.
Thus if $A$ is infinite, then for all $n \geq {n}_{r}\[\lc A \cap {I}_{n} \rc < \lc A \cap {I}_{n + 1} \rc\]$, as needed for clause (1) of Definition \ref{def:induce}.
Next, take $A \in {\A}_{p}$ and $X \in {\DD}_{p, A}$.
Choose $Y \in {\F}_{p, A}$ with $Y \subset X$.
Again there is $r \in G \cap {R}_{A, Y}$.
Fix $n \geq {n}_{r}$.
There is $t \in G$ such that $t \preceq r$ and $n < {n}_{t}$.
Since $n \in {n}_{t} \setminus {n}_{r}$, by (7), $ {u}^{q, n}_{A \cap {I}_{q, n}} = {\left( {g}_{t}(n)  \right)}_{A \cap {f}_{t}(n)} \subset {\Phi}_{r}(A) \subset Y \subset X$.
So $\forallbutfin n \in \omega \[{u}^{q, n}_{A \cap {I}_{q, n}} \subset X \]$, as needed.
Finally take $A, B \in {\A}_{p}$ with $A \; {\subset}^{\ast} \; B$.
${\F}_{p, A}$ and ${\F}_{p, B}$ are non-empty.
Take ${Y}_{0} \in  {\F}_{p, A}$ and ${Y}_{1} \in {\F}_{p, B}$.
Since ${R}_{A, {Y}_{0}}$ and ${R}_{B, {Y}_{1}}$ are dense open sets met by $G$, we can find $r \in G \cap {R}_{A, {Y}_{0}} \cap {R}_{B, {Y}_{1}}$.
Then $A, B \in {F}_{r}$ and ${n}_{r} \in \omega$.
Fix $n \geq {n}_{r}$.
Then there is $t \in G$ such that $t \preceq r$ and $n < {n}_{t}$.
Since $n \in {n}_{t} \setminus {n}_{r}$, by (8), ${\pi}_{p, B, A} \restrict \left( {u}^{q, n}_{B \cap {I}_{q, n}} \right) = {\pi}_{p, B, A} \restrict \left( {\left( {g}_{t}(n) \right)}_{B \cap {f}_{t}(n)} \right) = {\pi}_{{g}_{t}(n), B \cap {f}_{t}(n), A \cap {f}_{t}(n)} = {\pi}_{q, B \cap {I}_{q, n}, A \cap {I}_{q, n}}$.
Therefore
\begin{align*}
 \forallbutfin n \in \omega \[{\pi}_{p, B, A} \restrict \left( {u}^{q, n}_{B \cap {I}_{q, n}} \right) = {\pi}_{q, B \cap {I}_{q, n}, A \cap {I}_{q, n}}\].
\end{align*}
This completes the verification that $q$ induces $p$ and hence also the proof of the lemma.
\end{proof}
\begin{Lemma} \label{lem:main1}
 Assume $\MA(\sigma-\textrm{centered})$
 For every $C \in \Pset(\omega)$, $\{p \in {\P}_{1}: \exists {C}^{\ast} \in {\A}_{p}\[C \; {=}^{\ast} \; {C}^{\ast}\] \}$ is dense in ${\P}_{1}$.
\end{Lemma}
\begin{proof}
Fix $p \in {\P}_{1}$.
If $\exists A \in {\A}_{p}\[A \; {=}^{\ast} \; C\]$, then there is nothing to do.
So assume that $\forall A \in {\A}_{p}\[A \; {\neq}^{\ast} \; C\]$.
Since $0, \omega \in {\A}_{p}$ this implies that both $C$ and $\omega \setminus C$ are infinite.
Let $\BB$ denote the Boolean subalgebra of $\Pset(\omega)$ generated by ${\A}_{p}$.
For each $A \in {\A}_{p}$ choose a family ${\F}_{p, A} \subset {\DD}_{p, A}$ as in (2) of Definition \ref{def:0condition}.
Let $\Ra$ be the poset defined in the proof of Lemma \ref{lem:rep} (with respect to the fixed condition $p$).
Let $\preceq$ also be as in the proof of Lemma \ref{lem:rep}.
We define a new ordering on $\Ra$.
For $r, s \in \Ra$, $s \trianglelefteq r$ iff $s \preceq r$ and
\begin{enumerate}
 \item
 let ${\BB}^{+}_{r}$ denote the Boolean subalgebra of $\Pset(\omega)$ generated by ${F}_{r} \cup \{C\}$; then for any $A \in {\BB}^{+}_{r}$, $\forall n \in {n}_{s} \setminus {n}_{r}\[A \cap {f}_{s}(n) \neq 0 \ \text{iff} \ A \ \text{is infinite}\]$;
 \item
 for each infinite $A \in {\BB}^{+}_{r}$, 
 \begin{align*}
 \forall n \in {n}_{s}\[n + 1 \in {n}_{s} \setminus {n}_{r} \implies \lc A \cap {f}_{s}(n) \rc < \lc A \cap {f}_{s}(n + 1)\rc\].
 \end{align*}
\end{enumerate}
Then it is easy to check that $\langle \Ra, \trianglelefteq \rangle$ is a $\sigma$-centered poset.
Moreover for each $A \in {\A}_{p}$ and $Y \in {\F}_{p, A}$ let ${R}_{A, Y} = \{s \in \Ra: A \in {F}_{s} \wedge {\Phi}_{s}(A) \subset Y\}$; then it is easy to check that ${R}_{A, Y}$ is dense open in $\langle \Ra, \trianglelefteq \rangle$.
Now we check the following claim.
\begin{Claim} \label{claim:main1}
 For each $n \in \omega$, ${R}_{n} = \{s \in \Ra: n < {n}_{s}\}$ is dense open in $\langle \Ra, \trianglelefteq \rangle$.
\end{Claim}
\begin{proof}
It is easy to check that ${R}_{n}$ is open in $\langle \Ra, \trianglelefteq \rangle$.
The proof that it is dense is by induction on $n$.
Fix $n$ and suppose that the claim holds for all $m < n$.
Take $r \in \Ra$.
By the inductive hypothesis and by the openness of the ${R}_{m}$ for $m < n$, we may assume that $n \subset {n}_{r}$.
If $n < {n}_{r}$, then there is nothing to do.
So we assume $n = {n}_{r}$ and define $s$ so that ${n}_{s} = n + 1$.
Also $0, \omega \in {\A}_{p}$ and ${\F}_{p, 0}$ and ${\F}_{p, \omega}$ are non-empty.
If ${Y}_{0} \in {\F}_{p, 0}$ and ${Y}_{1} \in {\F}_{p, \omega}$, then ${R}_{0, {Y}_{0}}$ and ${R}_{\omega, {Y}_{1}}$ are dense open in $\langle \Ra, \trianglelefteq \rangle$, and so we may assume that $0, \omega \in {F}_{r}$.
Since ${\BB}^{+}_{r}$ is finite, we can find a finite non-empty ${f}_{s}(n) \subset \omega$ such that:
\begin{enumerate}
 \item[(3)]
 for every finite $A \in {\BB}^{+}_{r}$, $A \cap {f}_{s}(n) = 0$;
 \item[(4)]
 for every infinite $A \in {\BB}^{+}_{r}$, $A \cap {f}_{s}(n) \neq 0$;
 \item[(5)]
 if $n > 0$, then $\min({f}_{s}(n)) > \max({f}_{r}(n - 1))$ and for every infinite $A \in {\BB}^{+}_{r}$, $\lc {f}_{s}(n) \cap A\rc > \lc {f}_{r}(n - 1) \cap A \rc$.
\end{enumerate}
By the Representation Lemma fix $q \in \Q$ that induces $p$.
Let ${\BB}_{r}$ be the Boolean subalgebra of $\Pset(\omega)$ generated by ${F}_{r}$.
As ${\BB}_{r}$ is a finite subset of $\BB$ and ${F}_{r}$ is a finite subset of ${\A}_{p}$, we can find $m \in \omega$ such that the following hold:
\begin{enumerate}
 \item[(6)]
 for each finite $A \in {\BB}_{r}$, $A \cap {I}_{q, m} = 0$; for each infinite $A \in {\BB}_{r}$, $A \cap {I}_{q, m} \neq 0$; moreover for each infinite $A \in {\BB}_{r}$, $\lc A \cap {I}_{q, m}\rc \geq 2 \lc A \cap {f}_{s}(n)\rc$;
 \item[(7)]
 for each $A \in {F}_{r}$, ${u}^{q, m}_{A \cap {I}_{q, m}} \subset {\Phi}_{r}(A)$;
 \item[(8)]
 for each $A, B \in {F}_{r}$, if $A \; {\subset}^{\ast} \; B$, then ${\pi}_{p, B, A} \restrict {u}^{q, m}_{B \cap {I}_{q, m}} = {\pi}_{q, B \cap {I}_{q, m}, A \cap {I}_{q, m}}$;
 \item[(9)]
 if $n > 0$, then $\nor({u}^{q, m}) > \nor({g}_{r}(n - 1))$ and for each $a \subset {f}_{r}(n - 1)$ and each $b \subset {I}_{q, m}$, $\min({u}^{q, m}_{b}) > \max({\left( {g}_{r}(n - 1) \right)}_{a})$.
\end{enumerate}
Let $\{{A}_{0}, \dotsc, {A}_{l + 1}\}$ enumerate the elements of ${F}_{r} \cup \{C\}$, with $\{{A}_{0}, \dotsc, {A}_{l}\}$ being an enumeration of ${F}_{r}$ and ${A}_{l + 1} = C$.
For each $\sigma \in {2}^{l + 2}$ define the set ${b}_{\sigma} = \left( \bigcap \{{A}_{i}: \sigma(i) = 0 \} \right) \cap \left( \bigcap \{\omega \setminus {A}_{i}: \sigma(i) = 1 \} \right)$ (in this definition $\bigcap 0 = \omega$).
It is clear that each ${b}_{\sigma} \in {\BB}^{+}_{r}$.
For each $\tau \in {2}^{l + 1}$ define ${b}_{\tau} = \left( \bigcap \{{A}_{i}: \tau(i) = 0 \} \right) \cap \left( \bigcap \{\omega \setminus {A}_{i}: \tau(i) = 1 \} \right)$.
Note that each ${b}_{\tau} \in {\BB}_{r}$.
Let $T = \{\sigma \in {2}^{l + 2}: {b}_{\sigma} \ \text{is infinite}\}$ and let $S = \{\tau \in {2}^{l + 1}: {b}_{\tau} \ \text{is infinite}\}$.
If $\sigma \in T$, then $\sigma \restrict l + 1 \in S$.
Also if $\tau \in S$, then at least one of ${\tau}^{\frown}{\langle 0 \rangle}$ or ${\tau}^{\frown}{\langle 1 \rangle}$ is in $T$.
For each $\tau \in S$, by (6), $\lc {b}_{\tau} \cap {I}_{q, m} \rc \geq 2\lc {b}_{\tau} \cap {f}_{s}(n)\rc$.
So we can find disjoint sets ${b}^{0}_{\tau}$ and ${b}^{1}_{\tau}$ such that $\lc {b}^{0}_{\tau} \rc \geq \lc {b}_{\tau} \cap {f}_{s}(n) \rc$, $\lc {b}^{1}_{\tau} \rc \geq \lc {b}_{\tau} \cap {f}_{s}(n) \rc$, and ${b}_{\tau} \cap {I}_{q, m} = {b}^{0}_{\tau} \cup {b}^{1}_{\tau}$.
For each $\sigma \in T$ define a set ${c}_{\sigma}$ as follows.
If both ${\left( \sigma \restrict l + 1 \right)}^{\frown}{\langle 0 \rangle}$ and ${\left( \sigma \restrict l + 1 \right)}^{\frown}{\langle 1 \rangle}$ are members of $T$, then ${c}_{\sigma} = {b}^{\sigma(l + 1)}_{\left( \sigma \restrict l + 1 \right)}$.
Otherwise ${c}_{\sigma} = {b}_{\left( \sigma \restrict l + 1 \right)} \cap {I}_{q, m}$.
It is easy to check that ${f}_{s}(n) = {\bigcup}_{\sigma \in T}({b}_{\sigma} \cap {f}_{s}(n))$ and that ${I}_{q, m} = {\bigcup}_{\sigma \in T}{({c}_{\sigma} \cap {I}_{q, m})}$.
Also for each $\sigma, \sigma' \in T$, if $\sigma \neq \sigma'$, then ${c}_{\sigma} \cap {c}_{\sigma'} = 0$ and ${b}_{\sigma} \cap {b}_{\sigma'} = 0$.
Moreover for each $\sigma \in T$, $\lc {b}_{\sigma} \cap {f}_{s}(n) \rc \leq \lc {c}_{\sigma} \cap {I}_{q, m} \rc$, and ${b}_{\sigma} \cap {f}_{s}(n) \neq 0$.
So there is an onto map $h: {I}_{q, m} \rightarrow {f}_{s}(n)$ such that $\forall \sigma \in T \[{h}^{-1}({b}_{\sigma} \cap {f}_{s}(n)) = {c}_{\sigma} \cap {I}_{q, m}\]$.
For each $0 \leq i \leq l$, let ${T}_{i} = \{\sigma \in T: \sigma(i) = 0\}$.
It is easy to check that for each $0 \leq i \leq l$, ${A}_{i} \cap {f}_{s}(n) = {\bigcup}_{\sigma \in {T}_{i}}\left( {b}_{\sigma} \cap {f}_{s}(n)\right)$ and ${A}_{i} \cap {I}_{q, m} = {\bigcup}_{\sigma \in {T}_{i}}\left( {c}_{\sigma} \cap {I}_{q, m} \right)$.
Therefore for any $0 \leq i \leq l$, ${h}^{-1}({A}_{i} \cap {f}_{s}(n)) = {\bigcup}_{\sigma \in {T}_{i}}\left( {h}^{-1}\left( {b}_{\sigma} \cap {f}_{s}(n) \right)\right) = {\bigcup}_{\sigma \in {T}_{i}}\left( {c}_{\sigma} \cap {I}_{q, m} \right) = {A}_{i} \cap {I}_{q, m}$.
Define ${g}_{s}(n) = h\[{u}^{q, m}\]$.
Then ${g}_{s}(n)$ is a creature acting on ${f}_{s}(n)$ and if $n > 0$, then $\nor({g}_{s}(n)) > \nor({g}_{r}(n - 1))$.
Also if $a \subset {f}_{s}(n)$, ${\left( {g}_{s}(n) \right)}_{a} = {u}^{q, m}_{{h}^{-1}(a)} \subset \omega$ such that if $n > 0$, then for all $x \subset {f}_{r}(n - 1)$, $\max({\left({g}_{r}(n - 1)\right)}_{x}) < \min({\left( {g}_{s}(n)\right)}_{a})$.
Therefore if we let ${n}_{s} = n + 1$, ${f}_{s} = {{f}_{r}}^{\frown}{\langle {f}_{s}(n) \rangle}$, ${g}_{s} = {{g}_{r}}^{\frown}{\langle {g}_{s}(n)\rangle}$, ${F}_{s} = {F}_{r}$, and ${\Phi}_{s} = {\Phi}_{r}$, then $s = \langle {f}_{s}, {g}_{s}, {F}_{s}, {\Phi}_{s} \rangle$ is a member of $\Ra$.
We check that $s \trianglelefteq r$.
Clause (1) follows from (3) and (4), while (2) is a consequence of (5).
Next, to see that $s \preceq r$, note that (4) of Lemma \ref{lem:rep} is obvious, while (5) of Lemma \ref{lem:rep} follows from (1).
(6) of Lemma \ref{lem:rep} is by (2).
Next, take $A \in {F}_{r}$.
Then $A = {A}_{i}$ for some $0 \leq i \leq l$.
So by (7) ${\left( {g}_{s}(n) \right)}_{\left( A \cap {f}_{s}(n) \right)} = {u}^{q, m}_{{h}^{-1}(A \cap {f}_{s}(n))} = {u}^{q, m}_{A \cap {I}_{q, m}} \subset {\Phi}_{r}(A)$.
Finally take $A, B \in {F}_{r}$ and suppose $A \; {\subset}^{\ast} \; B$.
Note that $A \setminus B$ is a finite member of ${\BB}_{r}$.
So ${f}_{s}(n) \cap (A \setminus B) = 0$.
Hence $A \cap {f}_{s}(n) \subset B \cap {f}_{s}(n) \subset {f}_{s}(n)$.
Therefore ${\pi}_{{g}_{s}(n), B \cap {f}_{s}(n), A \cap {f}_{s}(n)}$ is defined and is equal to ${\pi}_{q, {h}^{-1}(B \cap {f}_{s}(n)), {h}^{-1}(A \cap {f}_{s}(n))}$, which in turn equals ${\pi}_{q, B \cap {I}_{q, m}, A \cap {I}_{q, m}}$.
By (8), ${\pi}_{q, B \cap {I}_{q, m}, A \cap {I}_{q, m}} = {\pi}_{p, B, A} \restrict {u}^{q, m}_{B \cap {I}_{q, m}} = {\pi}_{p, B, A} \restrict \left( {\left( {g}_{s}(n) \right)}_{B \cap {f}_{s}(n)} \right)$ because ${\left( {g}_{s}(n) \right)}_{B \cap {f}_{s}(n)} = {u}^{q, m}_{{h}^{-1}(B \cap {f}_{s}(n))} = {u}^{q, m}_{B \cap {I}_{q, m}}$.
This concludes the verification that $s \trianglelefteq r$ and hence the proof of the claim.
\end{proof}
Let $G \subset \Ra$ be a filter meeting all the dense open sets in $\{{R}_{n}: n \in \omega\} \cup \{{R}_{A, Y}: A \in {\A}_{p} \wedge Y \in {\F}_{p, A}\}$.
Let $I = {\bigcup}_{r \in G}{{f}_{r}}$ and $U = {\bigcup}_{r \in G}{{g}_{r}}$, and let ${q}_{0} = \langle I, U \rangle$.
Then ${q}_{0} \in \Q$.
Let ${\A}_{{p}_{0}} = {\A}_{p} \cup \{C\}$.
Then ${\A}_{p} \subset {\A}_{{p}_{0}} \subset \Pset(\omega)$, $\lc {\A}_{{p}_{0}} \rc < \c$, and $\forall A, B \in {\A}_{{p}_{0}}\[A \neq B \implies A \; {\neq}^{\ast} \; B\]$.
Let ${\BB}_{0}$ be the Boolean subalgebra of $\Pset(\omega)$ generated by ${\A}_{{p}_{0}}$.
Let $A$ be an infinite member of ${\BB}_{0}$.
There is a finite set $F \subset {\A}_{p}$ such that $A$ is in the Boolean subalgebra of $\Pset(\omega)$ generated by $F \cup \{C\}$.
Fix $r \in G$ such that $F \subset {F}_{r}$.
Then $A$ is an infinite member of ${\BB}^{+}_{r}$.
For any $n \geq {n}_{r}$, $\lc A \cap {I}_{{q}_{0}, n} \rc < \lc A \cap {I}_{{q}_{0}, n + 1}\rc$ because of (2).
Therefore, for any infinite $A \in {\BB}_{0}$, $\forallbutfin n \in \omega \[\lc A \cap {I}_{{q}_{0}, n} \rc < \lc A \cap {I}_{{q}_{0}, n + 1}\rc\]$.
It is also easy to see that ${q}_{0}$ induces $p$.
Now for each $A \in {\A}_{{p}_{0}}$, let ${X}_{A} = {\bigcup}_{n \in \omega} {u}^{{q}_{0}, n}_{{I}_{{q}_{0}, n} \cap A}$ and let ${\DD}_{{p}_{0}, A} = \{a \subset \omega: {X}_{A} \; {\subset}^{\ast} \; a \}$.
Put ${\D}_{{p}_{0}} = \langle {\DD}_{{p}_{0}, A}: A \in {\A}_{{p}_{0}}\rangle$.
For $A, B \in {\A}_{{p}_{0}}$ with $A \; {\subset}^{\ast} \; B$, if $A, B \in {\A}_{p}$, then define ${\pi}_{{p}_{0}, B, A} = {\pi}_{p, B, A}$.
If either $A$ or $B$ belongs to ${\A}_{{p}_{0}} \setminus {\A}_{p}$, then define ${\pi}_{{p}_{0}, B, A}: \omega \rightarrow \omega$ as follows.
Given $k \in \omega$, if $k \in {X}_{B}$, then there is a unique $n \in \omega$ such that $k \in {u}^{{q}_{0}, n}_{{I}_{{q}_{0}, n} \cap B}$.
If $A \cap {I}_{{q}_{0}, n} \subset B \cap {I}_{{q}_{0}, n}$, then ${\pi}_{{p}_{0}, B, A}(k) = {\pi}_{{q}_{0}, B \cap {I}_{{q}_{0}, n}, A \cap {I}_{{q}_{0}, n}}(k)$.
If either $A \cap {I}_{{q}_{0}, n} \not\subset B \cap {I}_{{q}_{0}, n}$ or if $k \notin {X}_{B}$, then put ${\pi}_{{p}_{0}, B, A}(k)= 0$.
Let ${\C}_{{p}_{0}} = \langle {\pi}_{{p}_{0}, B, A}: A, B \in {\A}_{{p}_{0}} \wedge A \; {\subset}^{\ast} \; B\rangle$ and let ${p}_{0} = \langle {\A}_{{p}_{0}}, {\C}_{{p}_{0}}, {\D}_{{p}_{0}}\rangle$.
Then it is not hard to see that ${p}_{0} \in {\P}_{0}$, ${p}_{0} \leq p$, and that ${q}_{0}$ induces ${p}_{0}$.
Hence ${q}_{0}$ also induces any ${p}_{1} \in {\P}_{0}$ with ${p}_{0} \leq {p}_{1}$.
So ${p}_{0} \in {\P}_{1}$ and ${p}_{0} \leq p$.
As $C \in {\A}_{{p}_{0}}$, this concludes the proof of the lemma.
\end{proof}
\begin{remark} \label{rem:key}
We now make some simple observations that will be useful for the remaining part of the proof.
Suppose $q \in \Q$.
Suppose $\langle {k}_{n}: n \in \omega \rangle \subset \omega$ is a sequence such that $\forall n \in \omega \[{k}_{n} < {k}_{n + 1}\]$.
For each $n \in \omega$, put ${I}_{{q}_{0}, n} = {I}_{q, {k}_{n}}$.
Suppose also that for each $n \in \omega$, we are given ${u}^{{q}_{0}, n} \in \Sigma({u}^{q, {k}_{n}})$ in such a way that for all $n \in \omega$, $\nor({u}^{{q}_{0}, n}) < \nor({u}^{{q}_{0}, n + 1})$.
Then if we let ${I}_{{q}_{0}} = \langle {I}_{{q}_{0}, n}: n \in \omega \rangle$, ${U}_{{q}_{0}} =\langle {u}^{{q}_{0}, n}: n \in \omega \rangle$, and ${q}_{0} = \langle {I}_{{q}_{0}}, {U}_{{q}_{0}}\rangle$, then ${q}_{0} \in \Q$.
Moreover, if $p \in {\P}_{0}$ and $q$ induces $p$, then ${q}_{0}$ also induces $p$.
We can now define ${p}_{0}$ using $p$ and ${q}_{0}$ as follows.
Put ${\A}_{{p}_{0}} = {\A}_{p}$.
For each $A \in {\A}_{{p}_{0}}$, let ${X}_{A} = {\bigcup}_{n \in \omega} {u}^{{q}_{0}, n}_{{I}_{{q}_{0}, n} \cap A}$ and let ${\DD}_{{p}_{0}, A} = \{a \subset \omega: {X}_{A} \; {\subset}^{\ast} \; a\}$.
Put ${\D}_{{p}_{0}} = \langle {\DD}_{{p}_{0}, A}: A \in {\A}_{{p}_{0}}\rangle$.
Given $A, B \in {\A}_{{p}_{0}}$ with $A \; {\subset}^{\ast} \; B$, set ${\pi}_{{p}_{0}, B, A} = {\pi}_{p, B, A}$.
Define ${\C}_{{p}_{0}} = \langle {\pi}_{{p}_{0}, B, A}: A, B \in {\A}_{{p}_{0}} \wedge A \; {\subset}^{\ast} \; B \rangle$ and ${p}_{0} = \langle {\A}_{{p}_{0}}, {\C}_{{p}_{0}}, {\D}_{{p}_{0}} \rangle$.
Then ${p}_{0} \in {\P}_{0}$, ${p}_{0} \leq p$, and ${q}_{0}$ induces ${p}_{0}$.
Therefore, ${q}_{0}$ also induces any ${p}_{1} \in {\P}_{0}$ with ${p}_{0} \leq {p}_{1}$.
Hence ${p}_{0} \in {\P}_{1}$.
\end{remark}
\begin{Lemma} \label{lem:ultra}
 Suppose $p \in {\P}_{1}$ and $A \in {\A}_{p}$.
 Let $b \subset \omega$.
 There exists ${p}_{0} \in {\P}_{1}$, ${p}_{0} \leq p$ such that either $b \in {\DD}_{{p}_{0}, A}$ or $\omega \setminus b \in {\DD}_{{p}_{0}, A}$.
\end{Lemma}
\begin{proof}
 Let ${b}_{0} = b$ and ${b}_{1} = \omega \setminus b$.
 By the Representation Lemma fix $q \in \Q$ that induces $p$.
 Fix $n \geq 1$.
 Then $\nor({u}^{q, n}) \geq (n - 1) + 1$.
 We have that ${u}^{q, n}_{A \cap {I}_{q, n}} = \left( {u}^{q, n}_{A \cap {I}_{q, n}} \cap {b}_{0} \right) \cup \left( {u}^{q, n}_{A \cap {I}_{q, n}} \cap {b}_{1} \right)$.
 So there exists ${j}_{n} \in 2$ and ${v}^{n} \in \Sigma({u}^{q, n})$ such that $\nor({v}^{n}) \geq n - 1$ and ${v}^{n}_{A \cap {I}_{q, n}} \subset {u}^{q, n}_{A \cap {I}_{q, n}} \cap {b}_{{j}_{n}}$.
Clearly, there is $j \in 2$ such that $\{n \geq 1: {j}_{n} = j\}$ is infinite.
So it is possible to find a sequence $\langle {k}_{n}: n \in \omega \rangle \subset \omega$ such that for each $n \in \omega$, ${k}_{n} \geq 1$, ${j}_{{k}_{n}} = j$, ${k}_{n} < {k}_{n + 1}$, and $\nor({v}^{{k}_{n + 1}}) > \nor({v}^{{k}_{n}})$.
For each $n \in \omega$, let ${u}^{{q}_{0}, n} = {v}^{{k}_{n}} \in \Sigma({u}^{q, {k}_{n}})$.
Also $\nor({u}^{{q}_{0}, n}) < \nor({u}^{{q}_{0}, n + 1})$ holds for all $n \in \omega$.
Therefore if ${q}_{0}$ and ${p}_{0}$ are defined as in Remark \ref{rem:key}, then ${p}_{0} \in {\P}_{1}$ and ${p}_{0} \leq p$.
Moreover note that for each $n \in \omega$, ${u}^{{q}_{0}, n}_{{I}_{{q}_{0}, n} \cap A} = {v}^{{k}_{n}}_{{I}_{q, {k}_{n}} \cap A} \subset \left( {u}^{q, {k}_{n}}_{A \cap {I}_{q, {k}_{n}}} \cap {b}_{{j}_{{k}_{n}}} \right) \subset {b}_{j}$.
Hence ${X}_{A} \subset {b}_{j}$, whence ${b}_{j} \in {\DD}_{{p}_{0}, A}$, as needed.
\end{proof}
\begin{Def} \label{def:p-pointgame}
Let $\U$ be an ultrafilter on $\omega$.
The \emph{P-point game on} $\U$ is a two player game in which Players I and II alternatively choose sets ${a}_{n}$ and ${s}_{n}$ respectively, where ${a}_{n} \in \U$ and ${s}_{n} \in {\[{a}_{n}\]}^{< \omega}$.
Together they construct the sequence
\begin{align*}
{a}_{0}, {s}_{0}, {a}_{1}, {s}_{1}, \dotsc	\end{align*}
Player I wins iff ${\bigcup}_{n \in \omega}{{s}_{n}} \notin  \U$.
\end{Def}
A proof of the following useful characterization of P-points in terms of the P-point game can be found in Bartoszy{\'n}ski and Judah~\cite{BJ}.
\begin{Theorem}
An ultrafilter $\U$ is a P-point iff Player I does not have a winning strategy in the P-point game on $\U$.
\end{Theorem}
\begin{Lemma} \label{lem:localmap}
 Suppose $\VV$ is a P-point and $\U$ is any ultrafilter.
 Suppose $\phi: \VV \rightarrow \U$ is monotone and cofinal in $\U$.
 Then there exist $P \subset {\[\omega\]}^{< \omega} \setminus \{0\}$ and $f: P \rightarrow \omega$ such that the following things hold:
 \begin{enumerate}
  \item
  $\forall s, t \in P\[s \subset t \implies s = t\]$;
  \item
  $f$ is finite-to-one;
  \item
  $\forall a \in \VV\forall b \in \U\exists s \in P\[s \subset a \wedge f(s) \in b\]$.
 \end{enumerate}
\end{Lemma}
\begin{proof}
 Define $\psi: \Pset(\omega) \rightarrow \Pset(\omega)$ by $\psi(x) = \bigcap\{\phi(a): a \in \VV \wedge x \subset a\}$, for all $x \in \Pset(\omega)$.
 Note that $\psi$ is monotone.
 Also $\psi(0) = 0$.
 To see this, suppose for a contradiction that $k \in \psi(0)$.
 Then $\omega \setminus \{k\} \in \U$.
 Take $a \in \VV$ such that $\phi(a) \subset \omega \setminus \{k\}$.
 However since $k \in \psi(0)$, $k \in \phi(a)$, a contradiction.
 Now we define a strategy for Player I in the P-point game (on $\VV$) as follows.
 He first plays ${a}_{0} = \omega$.
 Given $n \in \omega$ and a partial play ${a}_{0}, {s}_{0}, \dotsc, {a}_{n}, {s}_{n}$, he considers $\Pset({\bigcup}_{i \leq n}{s}_{i})$.
 For each $s \in \Pset({\bigcup}_{i \leq n}{s}_{i})$, if $n \notin \psi(s)$, then he chooses ${a}_{n, s} \in \VV$ such that $s \subset {a}_{n, s}$ and yet $n \notin \phi({a}_{n, s})$.
 He plays 
 \begin{align*}
 {a}_{n + 1} = \left({a}_{n} \setminus {l}_{n} \right)\cap \left( \bigcap\{{a}_{n, s}: s \in \Pset({\bigcup}_{i \leq n}{s}_{i}) \wedge n \notin \psi(s)\} \right),
 \end{align*}
 where ${l}_{n} = \sup\{k + 1: k \in {\bigcup}_{i \leq n}{s}_{i} \} \in \omega$ (in this definition of ${a}_{n + 1}$, $\bigcap 0$ is taken to be $\omega$).
 Since this is not a winning strategy for Player I, there is a run ${a}_{0}, {s}_{0}, \dotsc, {a}_{n}, {s}_{n}, \dotsc$ of the P-point game in which he implements this strategy and looses.
 So $b = {\bigcup}_{n \in \omega}{{s}_{n}} \in \VV$.
 Note that by the definition of the strategy, $\forall n \in \omega\[{a}_{n + 1} \subset {a}_{n}\]$.
 Also since ${s}_{n + 1} \subset {a}_{n + 1}$, if $k \in {s}_{n}$ and $k' \in {s}_{n + 1}$, then $k < k'$.
 Let $P = \{t \in {\[b\]}^{< \omega}: \psi(t) \neq 0 \wedge \forall s \subsetneq t \[\psi(s) = 0\]\}$.
 Since $\psi(0) = 0$, $P \subset {\[\omega\]}^{< \omega} \setminus \{0\}$.
 It is clear that $P$ satisfies (1) by definition.
 Define $f: P \rightarrow \omega$ by $f(t) = \min(\psi(t))$, for all $t \in P$
 Now we claim the following.
 \begin{Claim} \label{claim:localmap1}
  For any $n \in \omega$ and any $c \in \VV$, if $c \subset b$ and $n \in \phi(c)$, then $n \in \psi\left(c \cap \left( {\bigcup}_{i \leq n}{s}_{i} \right) \right)$.
 \end{Claim}
 \begin{proof}
 Suppose not.
 Let $s = c \cap \left( {\bigcup}_{i \leq n}{s}_{i} \right)$.
 Since $n \notin \psi(s)$, ${a}_{n, s}$ exists and ${a}_{n + 1} \subset {a}_{n, s}$.
 Moreover, for any $m \geq n + 1$, ${s}_{m} \subset {a}_{m} \subset {a}_{n + 1} \subset {a}_{n, s}$.
 Therefore, $c = c \cap b = {\bigcup}_{m \in \omega} \left( c \cap {s}_{m} \right) = s \cup \left( {\bigcup}_{m \geq n + 1} \left(c \cap {s}_{m} \right)\right) \subset {a}_{n, s}$.
 Hence $\phi(c) \subset \phi({a}_{n, s})$, whence $n \notin \phi(c)$.
 \end{proof}
 Both (2) and (3) easily follow from Claim \ref{claim:localmap1}.
 For (2), fix $n \in \omega$ and suppose $t \in P$ is such that $f(t) = n$.
 Then $n \in \psi(t)$.
 Consider $c = t \cup \left( {\bigcup}_{m \geq n + 1} {s}_{m}\right)$.
 It is clear that $c \in \VV$, $t \subset c$, and $c \subset b$.
 So $n \in \phi(c)$.
 So by Claim \ref{claim:localmap1}, $n \in \psi\left(c \cap \left( {\bigcup}_{m \leq n}{s}_{m} \right)\right) = \psi\left(t \cap \left( {\bigcup}_{m \leq n}{s}_{m} \right)\right)$.
 Since $t \in P$, this implies that $t \cap \left( {\bigcup}_{m \leq n}{s}_{m} \right) = t$.
 Thus ${f}^{-1}(\{n\}) \subset \Pset({\bigcup}_{m \leq n}{s}_{m})$, which is finite.
 
 Next for (3), fix $c \in \VV$ and $d \in \U$.
 Let $e \in \VV$ be such that $\phi(e) \subset d$.
 Then $b \cap c \cap e \in \VV$, $\phi(b \cap c \cap e) \in \U$.
 So $\phi(b \cap c \cap e) \neq 0$.
 If $n \in \phi(b \cap c \cap e)$, then $n \in \psi(u)$, where $u = \left( b \cap c \cap e \right) \cap \left( {\bigcup}_{m \leq n}{s}_{m}\right)$.
 Thus $\psi(u) \neq 0$, and we may find $t \subset u$ that is $\subset$-minimal w.\@r.\@t.\@ the property that $\psi(t) \neq 0$.
 Then $t \in P$ and $t \subset u \subset b \cap c \cap e \subset c$, and $f(t) \in \psi(t)$.
 Since $t \subset e$ and $e \in \VV$, $f(t) \in \phi(e) \subset d$, as needed.
\end{proof}
\begin{Lemma} \label{lem:killingtukey}
Assume $\MA(\sigma-\textrm{centered})$.
 Suppose $p \in {\P}_{1}$.
 Suppose $A, B \in {\A}_{p}$ with $B \; {\not\subset}^{\ast} \; A$.
 Suppose that $P \subset {\[\omega\]}^{< \omega} \setminus \{0\}$ and $f: P \rightarrow \omega$ satisfy (1)-(2) of Lemma \ref{lem:localmap}.
 Then there exists ${p}_{0} \in {\P}_{1}$ such that ${p}_{0} \leq p$ and there exist sets $X \in {\DD}_{{p}_{0}, A}$ and $Y \in {\DD}_{{p}_{0}, B}$ such that $\forall s \in P\[s \subset X \implies f(s) \notin Y\]$.
\end{Lemma}
\begin{proof}
 Fix $q \in \Q$ that induces $p$.
 There is a $m \in \omega$ such that 
 \begin{align*}
 \forall n \geq m\[\lc \left( B \setminus A \right) \cap {I}_{q, n} \rc < \lc \left( B \setminus A \right) \cap {I}_{q, n + 1} \rc\]
 \end{align*}
 because $B \setminus A$ is an infinite member of the Boolean subalgebra of $\Pset(\omega)$ generated by ${\A}_{p}$.
 For each $n \in \omega$, consider ${\bigcup}_{k \leq n}{u}^{q, k}_{{I}_{q, k} \cap A}$.
 This is a finite subset of $\omega$.
 So $l(n) = \sup\left\{f(s): s \in P \wedge s \subset {\bigcup}_{k \leq n}{u}^{q, k}_{{I}_{q, k} \cap A}\right\} < \omega$.
 Similarly ${\bigcup}_{k \leq n}{u}^{q, k}_{{I}_{q, k} \cap B}$ is a finite subset of $\omega$.
 By (2) of Lemma \ref{lem:localmap}, for each $i \in {\bigcup}_{k \leq n}{u}^{q, k}_{{I}_{q, k} \cap B}$, $\bigcup \left( {f}^{-1}(\{i\}) \right)$ is a finite subset of $\omega$.
 So ${l}^{+}(n) = \sup\left(\bigcup \left\{\bigcup \left({f}^{-1}(\{i\})\right): i \in {\bigcup}_{k \leq n}{u}^{q, k}_{{I}_{q, k} \cap B}\right\}\right) < \omega$.
 Build two sequences $\langle {k}_{n}: n \in \omega \rangle$ and $\langle {u}^{{q}_{0}, n}: n \in \omega \rangle$ such that for each $n \in \omega$:
 \begin{enumerate}
  \item
  ${k}_{n} \in \omega$ and ${u}^{{q}_{0}, n} \in \Sigma({u}^{q, {k}_{n}})$;
  \item
  $\forall j < n\[{k}_{j} < {k}_{n}\]$ and $\forall j < n\[\nor({u}^{{q}_{0}, j}) < \nor({u}^{{q}_{0}, n})\]$;
  \item
  for any $s \subset \left({\bigcup}_{j < n}{u}^{{q}_{0}, j}_{{I}_{q, {k}_{j}} \cap A} \right)$ and any $t \subset \left({u}^{{q}_{0}, n}_{{I}_{q, {k}_{n}} \cap A} \right)$, if $s \cup t \in P$, then $f(s \cup t) \notin {u}^{{q}_{0}, n}_{{I}_{q, {k}_{n}} \cap B}$;
  \item
  $\forall j < n\[l({k}_{j}) < \min\left({u}^{{q}_{0}, n}_{{I}_{q, {k}_{n}} \cap B}\right)\]$ and $\forall j < n\[{l}^{+}({k}_{j}) < \min\left({u}^{{q}_{0}, n}_{{I}_{q, {k}_{n}} \cap A}\right)\]$.
 \end{enumerate}
 Suppose for a moment that such a sequence can be built.
 Let ${q}_{0}$ and ${p}_{0}$ be defined as in Remark \ref{rem:key}.
 Then ${p}_{0} \in {\P}_{1}$ and ${p}_{0} \leq p$.
 Let ${X}_{A} = {\bigcup}_{n \in \omega}{u}^{{q}_{0}, n}_{{I}_{q, {k}_{n}} \cap A}$ and ${X}_{B} = {\bigcup}_{n \in \omega}{u}^{{q}_{0}, n}_{{I}_{q, {k}_{n}} \cap B}$.
 Note that ${X}_{A} \in {\DD}_{{p}_{0}, A}$ and ${X}_{B} \in {\DD}_{{p}_{0}, B}$.
 Suppose towards a contradiction that there exists ${s}^{\ast} \in P$ such that ${s}^{\ast} \subset {X}_{A}$ and $f({s}^{\ast}) \in {X}_{B}$.
 As ${s}^{\ast}$ is a non-empty finite subset of $\omega$, $\max({s}^{\ast})$ exists and there exists a unique $n \in \omega$ such that $\max({s}^{\ast}) \in {u}^{{q}_{0}, n}_{{I}_{q, {k}_{n}} \cap A}$.
 Then ${s}^{\ast} = s \cup t$, where $s = {s}^{\ast} \cap \left( {\bigcup}_{j < n} {u}^{{q}_{0}, j}_{{I}_{q, {k}_{j}} \cap A} \right)$ and $t = {s}^{\ast} \cap {u}^{{q}_{0}, n}_{{I}_{q, {k}_{n}} \cap A}$.
 By clause (3), $f({s}^{\ast}) \notin {u}^{{q}_{0}, n}_{{I}_{q, {k}_{n}} \cap B}$.
 By the definition of $l({k}_{n})$, $f({s}^{\ast}) \leq l({k}_{n})$.
 So by clause (4), $\forall {n}^{\ast} > n\[f({s}^{\ast}) \notin {u}^{{q}_{0}, {n}^{\ast}}_{{I}_{q, {k}_{{n}^{\ast}}} \cap B}\]$.
 So it must be that $f({s}^{\ast}) \in {u}^{{q}_{0}, j}_{{I}_{q, {k}_{j}} \cap B}$ for some $j < n$.
 But then $\max({s}^{\ast}) \leq {l}^{+}({k}_{j})$ contradicting clause (4).
 Therefore there is no ${s}^{\ast} \in P$ such that ${s}^{\ast} \subset {X}_{A}$ and $f({s}^{\ast}) \in {X}_{B}$.
 Hence ${p}_{0}$ is as required.
 
 To build the sequences $\langle {k}_{n}: n \in \omega \rangle$ and $\langle {u}^{{q}_{0}, n}: n \in \omega \rangle$ proceed as follows.
 Fix $n \in \omega$ and suppose that $\langle {k}_{j}: j < n \rangle$ and $\langle {u}^{{q}_{0}, j}: j < n \rangle$ are given.
 Let $M = \{m\} \cup \{{k}_{j}: j < n\} \cup \{\nor({u}^{{q}_{0}, j}) + 1: j < n\} \cup \{l({k}_{j}): j < n\} \cup \{{l}^{+}({k}_{j}): j < n\}$.
 $M$ is a finite non-empty subset of $\omega$.
 Let $k = \max(M) < \omega$.
 Let $x = \left({\bigcup}_{j < n}{u}^{{q}_{0}, j}_{{I}_{q, {k}_{j}} \cap A} \right)$.
 $x$ is a finite set.
 Put ${k}_{n} = k + {2}^{\lc x \rc} < \omega$.
 Note that ${k}_{n} > k \geq m$.
 Therefore $\left( B \setminus A \right) \cap {I}_{q, {k}_{n}} \neq 0$.
 So $B \cap {I}_{q, {k}_{n}} \not\subset A \cap {I}_{q, {k}_{n}}$.
 Also $\nor({u}^{q, {k}_{n}}) \geq {k}_{n} = k + {2}^{\lc x \rc}$.
 Let $\langle {s}_{i}: i < {2}^{\lc x \rc}\rangle$ enumerate all subsets of $x$.
 Now build a sequence $\langle {v}^{i}: i < {2}^{\lc x \rc} \rangle$ such that for each $i < {2}^{\lc x \rc}$:
 \begin{enumerate}
  \item[(5)]
  ${v}^{i} \in \Sigma({u}^{q, {k}_{n}})$ and $\nor({v}^{i}) \geq k + {2}^{\lc x \rc} - i - 1$;
  \item[(6)]
  $\forall {i}^{\ast} < i\[{v}^{i} \in \Sigma\left({v}^{{i}^{\ast}}\right)\]$;
  \item[(7)]
  for any $t \subset {v}^{i}_{{I}_{q, {k}_{n}} \cap A}$, if ${s}_{i} \cup t \in P$, then $f({s}_{i} \cup t) \notin {v}^{i}_{{I}_{q, {k}_{n}} \cap B}$.
 \end{enumerate}
 This sequence is constructed by induction on $i < {2}^{\lc x \rc}$.
 Fix $i < {2}^{\lc x \rc}$ and suppose that ${v}^{{i}^{\ast}}$ is given for all ${i}^{\ast} < i$.
 If $i > 0$, let $v = {v}^{i - 1}$, if $i = 0$, then let $v = {u}^{q, {k}_{n}}$.
 In either case $v \in \cre({I}_{q, {k}_{n}})$ and $\nor(v) \geq (k + {2}^{\lc x \rc} -i - 1) + 1$.
 Now ${v}_{{I}_{q, {k}_{n}} \cap B}$ is a non-empty set.
 Fix ${z}_{0} \in {v}_{{I}_{q, {k}_{n}} \cap B}$.
 Define a function $F: \Pset({v}_{{I}_{q, {k}_{n}} \cap A}) \rightarrow {v}_{{I}_{q, {k}_{n}} \cap B}$ as follows.
 Given $t \in \Pset({v}_{{I}_{q, {k}_{n}} \cap A})$, if ${s}_{i} \cup t \in P$ and $f({s}_{i} \cup t) \in {v}_{{I}_{q, {k}_{n}} \cap B}$, then let $F(t) = f({s}_{i} \cup t)$.
 Otherwise let $F(t) = {z}_{0}$.
 There exists ${v}^{i} \in \Sigma(v)$ with $\nor({v}^{i}) \geq k + {2}^{\lc x \rc} -i - 1$ such that $F''\Pset({v}^{i}_{{I}_{q, {k}_{n}} \cap A}) \cap {v}^{i}_{{I}_{q, {k}_{n}} \cap B} = 0$.
 It is clear that ${v}^{i}$ is as needed.
 
 Now let $i = {2}^{\lc x \rc} - 1 < {2}^{\lc x \rc}$ and define ${u}^{{q}_{0}, n} = {v}^{i}$.
 By (5), ${v}^{i} \in \Sigma({u}^{q, {k}_{n}})$, and so (1) is satisfied.
 For (2) note that $\nor({v}^{i}) \geq k + {2}^{\lc x \rc} - i - 1 = k \geq \nor({u}^{{q}_{0}, j}) + 1 > \nor({u}^{{q}_{0}, j})$, for all $j < n$.
 Next to check (3) fix $s \subset \left({\bigcup}_{j < n}{u}^{{q}_{0}, j}_{{I}_{q, {k}_{j}} \cap A} \right) = x$ and $t \subset \left({u}^{{q}_{0}, n}_{{I}_{q, {k}_{n}} \cap A} \right)$.
 Suppose $s \cup t \in P$.
 Then $s = {s}_{{i}^{\ast}}$ for some ${i}^{\ast} \leq i$.
 It follows from (6) that ${u}^{{q}_{0}, n}_{{I}_{q, {k}_{n}} \cap A} \subset {v}^{{i}^{\ast}}_{{I}_{q, {k}_{n}} \cap A}$ and ${u}^{{q}_{0}, n}_{{I}_{q, {k}_{n}} \cap B} \subset {v}^{{i}^{\ast}}_{{I}_{q, {k}_{n}} \cap B}$.
 So by (7) applied to ${i}^{\ast}$ we have that $f(s \cup t) \notin {u}^{{q}_{0}, n}_{{I}_{q, {k}_{n}} \cap B}$.
 Finally for (4) note that ${u}^{{q}_{0}, n}_{{I}_{q, {k}_{n}} \cap A} \subset {u}^{q, {k}_{n}}_{{I}_{q, {k}_{n}} \cap A}$ and ${u}^{{q}_{0}, n}_{{I}_{q, {k}_{n}} \cap B} \subset {u}^{q, {k}_{n}}_{{I}_{q, {k}_{n}} \cap B}$.
 So $\min\left( {u}^{{q}_{0}, n}_{{I}_{q, {k}_{n}} \cap A} \right) \geq \min\left( {u}^{q, {k}_{n}}_{{I}_{q, {k}_{n}} \cap A} \right) \geq {k}_{n} > k \geq {l}^{+}({k}_{j})$, for all $j < n$, and $\min\left( {u}^{{q}_{0}, n}_{{I}_{q, {k}_{n}} \cap B} \right) \geq \min\left( {u}^{q, {k}_{n}}_{{I}_{q, {k}_{n}} \cap B} \right) \geq {k}_{n} > k \geq l({k}_{j})$, for all $j < n$.
 Thus ${u}^{{q}_{0}, n}$ and ${k}_{n}$ are as required.
\end{proof}
The following lemma is easy to check and tells us what to do at limit stages of the final inductive construction.
We leave the proof to the reader.
\begin{Lemma} \label{lem:limitconditiondef}
 Assume $\MA(\sigma-\textrm{centered})$.
 Let $\delta < \c$ be a limit ordinal.
 Suppose $\langle {p}_{\alpha}: \alpha < \delta \rangle$ be a sequence of conditions in ${\P}_{0}$ such that $\forall \alpha \leq \beta < \delta \[{p}_{\beta} \leq {p}_{\alpha}\]$.
 Define ${\A}_{{p}_{\delta}} = {\bigcup}_{\alpha < \delta} {{\A}_{{p}_{\alpha}}}$.
 For any $A \in {\A}_{{p}_{\delta}}$ let ${\alpha}_{A} = \min\{\alpha < \delta: A \in {\A}_{{p}_{\alpha}}\}$.
 For $A \in {\A}_{{p}_{\delta}}$ define ${\DD}_{{p}_{\delta}, A} = {\bigcup}_{{\alpha}_{A} \leq \alpha < \delta} {{\DD}_{{p}_{\alpha}, A}}$, and define ${\D}_{{p}_{\delta}} = \langle {\DD}_{{p}_{\delta}, A}: A \in {\A}_{{p}_{\delta}}\rangle$.
 Given $A, B \in {\A}_{{p}_{\delta}}$ with $A \; {\subset}^{\ast} \; B$, let ${\alpha}_{A, B} = \max\{{\alpha}_{A}, {\alpha}_{B}\}$, and define ${\pi}_{{p}_{\delta}, B, A} = {\pi}_{{p}_{{\alpha}_{A, B}}, B, A}$.
 Define ${\C}_{{p}_{\delta}} = \langle {\pi}_{{p}_{\delta}, B, A}: A, B \in {\A}_{{p}_{\delta}} \wedge A \; {\subset}^{\ast} \; B \rangle$.
 Finally define ${p}_{\delta} = \langle {\A}_{{p}_{\delta}}, {\C}_{{p}_{\delta}}, {\D}_{{p}_{\delta}} \rangle$.
 Then ${p}_{\delta} \in {\P}_{0}$ and $\forall \alpha < \delta\[{p}_{\delta} \leq {p}_{\alpha}\]$.
\end{Lemma}
\begin{Lemma} \label{lem:cofomegalimitcondition}
Assume $\MA(\sigma-\textrm{centered})$.
 Let $\delta < \c$ be a limit ordinal with $\cf(\delta) = \omega$.
 Suppose $\langle {p}_{\alpha}: \alpha < \delta \rangle$ is a sequence of conditions in ${\P}_{1}$ such that $\forall \alpha \leq \beta < \delta\[{p}_{\beta} \leq {p}_{\alpha}\]$.
 Suppose ${p}_{\delta} \in {\P}_{0}$ is defined as in Lemma \ref{lem:limitconditiondef}.
 Then ${p}_{\delta} \in {\P}_{1}$.
\end{Lemma}
\begin{proof}
 Take a finitary $p' \in {\P}_{0}$ with ${p}_{\delta} \leq p'$.
 For each $A \in {\A}_{p'}$ let ${\alpha}_{A}$ be defined as in Lemma \ref{lem:limitconditiondef}.
 For each $A \in {\A}_{p'}$, ${\F}_{p', A}$ is non-empty and countable; let $\{{Y}_{A, n}: n \in \omega\}$ enumerate ${\F}_{p', A}$.
 For each $A \in {\A}_{p'}$ and $n \in \omega$ choose ${\alpha}_{A} \leq {\alpha}_{A, n} < \delta$ such that ${Y}_{A, n} \in {\DD}_{{p}_{{\alpha}_{A, n}}, A}$.
 Find a strictly increasing cofinal sequence $\langle {\alpha}_{n}: n \in \omega \rangle$ of elements of $\delta$ such that ${\A}_{p'} \subset {\A}_{{p}_{{\alpha}_{0}}}$ and $\forall A \in {\A}_{p'} \forall i < n \[{\alpha}_{A, i} < {\alpha}_{n}\]$.
 Define a standard sequence $q$ as follows.
 Fix $n \in \omega$ and suppose that ${I}_{q, m}$ and ${u}^{q, m}$ are given for all $m < n$.
 Choose ${q}_{n} \in \Q$ inducing ${p}_{{\alpha}_{n}}$.
 We now define six collections of natural numbers as follows.
 First, let ${\BB}_{p'}$ denote the Boolean subalgebra of $\Pset(\omega)$ generated by ${\A}_{p'}$.
 If $A$ is an infinite member of ${\BB}_{p'}$, then there exists ${k}_{A} \in \omega$ such that $\forall k \geq {k}_{A}\[\lc A \cap {I}_{{q}_{n}, k}\rc < \lc A \cap {I}_{{q}_{n}, k + 1}\rc\]$.
 Define $\sup\{{k}_{A} + \lc {I}_{q, m} \cap A \rc + 1: m < n\} = {l}_{A}$.
 Second, say $A \in {\A}_{p'}$ and $i < n$.
 Then there exists ${l}_{A, i} \in \omega$ such that $\forall k \geq {l}_{A, i}\[{u}^{{q}_{n}, k}_{A \cap {I}_{{q}_{n}, k}} \subset {Y}_{A, i}\]$.
 Third, say $A, B \in {\A}_{p'}$ with $A \; {\subset}^{\ast} \; B$.
 Then there exists ${l}_{A, B} \in \omega$ such that $\forall k \geq {l}_{A, B}\[{\pi}_{{p}_{{\alpha}_{n}}, B, A} \restrict {u}^{{q}_{n}, k}_{B \cap {I}_{{q}_{n}, k}} = {\pi}_{{q}_{n}, B \cap {I}_{{q}_{n}, k}, A \cap {I}_{{q}_{n}, k}}\]$.
 Observe that since ${p}_{\delta} \leq {p}_{{\alpha}_{n}}$ and ${p}_{\delta} \leq p'$, ${\pi}_{{p}_{{\alpha}_{n}}, B, A} = {\pi}_{p', B, A}$.
 Fourth, define ${l}_{0} = \sup\{ \max({I}_{q, m}) + 1: m < n\}$.
 Fifth, let $\sup\{\nor({u}^{q, m}) + 1: m < n\} = {l}_{1}$.
 Sixth, define ${l}_{2} = \sup\left\{ \max({u}^{q, m}_{a}) + 1: m < n \wedge a \in \Pset({I}_{q, m})\right\}$.
 Now consider $M = \{{l}_{A}: A \in {\BB}_{p'} \wedge A \ \text{is infinite}\} \cup \{{l}_{A, i}: A \in {\A}_{p'} \wedge i < n\} \cup \{{l}_{A, B}: A, B \in {\A}_{p'} \wedge A \; {\subset}^{\ast} \; B \} \cup \{{l}_{0}, {l}_{1}, {l}_{2}\}$.
 $M$ is a finite non-empty subset of $\omega$.
 Let $l = \max(M)$.
 Then $l \in \omega$.
 Put ${I}_{q, n} = {I}_{{q}_{n}, l}$ and ${u}^{q, n} = {u}^{{q}_{n}, l}$.
 This completes the definition of $q$.
 It is easy to see that $q \in \Q$ and that $q$ induces $p'$.
 Therefore ${p}_{\delta} \in {\P}_{1}$.
\end{proof}
\begin{Lemma} \label{lem:unccoflimitcondition}
 Assume $\MA(\sigma-\textrm{centered})$.
 Let $\delta < \c$ be a limit ordinal with $\cf(\delta) > \omega$.
 Suppose $\langle {p}_{\alpha}: \alpha < \delta \rangle$ is a sequence of conditions in ${\P}_{1}$ such that $\forall \alpha \leq \beta < \delta\[{p}_{\beta} \leq {p}_{\alpha}\]$.
 Suppose ${p}_{\delta} \in {\P}_{0}$ is defined as in Lemma \ref{lem:limitconditiondef}.
 Then ${p}_{\delta} \in {\P}_{1}$.
\end{Lemma}
\begin{proof}
 Take a finitary $p' \in {\P}_{0}$ with ${p}_{\delta} \leq p'$.
 Since $\cf(\delta) > \omega$, there is $\alpha < \delta$ such that ${p}_{\alpha} \leq p'$.
 There is a $q \in \Q$ such that $q$ induces ${p}_{\alpha}$. This $q$ also induces $p'$.
 Hence ${p}_{\delta} \in {\P}_{1}$. 
\end{proof}
We are now ready to prove the main theorem.
We construct a set $\XX$ of representatives for the equivalence classes in $\Pset(\omega) \slash \FIN$ and index the ultrafilters by members of $\XX$.
\begin{Theorem} \label{thm:main}
 Assume $\MA(\sigma-\textrm{centered})$.
 There exists a set $\XX \subset \Pset(\omega)$ and a sequence $\langle {\U}_{A}: A \in \XX \rangle$ such that the following hold:
 \begin{enumerate}
  \item
  $\forall A, B \in \XX \[A \neq B \implies A \; {\neq}^{\ast} \; B\]$ and $\forall C \in \Pset(\omega) \exists A \in \XX\[C \; {=}^{\ast} \; A\]$;
  \item
  for each $A \in \XX$, ${\U}_{A}$ is a P-point;
  \item
  $\forall A, B \in \XX\[A \; {\subset}^{\ast} \; B \implies {\U}_{A} \; {\leq}_{RK} \; {\U}_{B}\]$;
  \item
  $\forall A, B \in \XX \[B \; {\not\subset}^{\ast} \; A \implies {\U}_{B} \; {\nleq}_{T} \; {\U}_{A}\]$.
 \end{enumerate}
\end{Theorem}
\begin{proof}
 Let $\c = {T}_{0} \cup {T}_{1} \cup {T}_{2} \cup {T}_{3}$ be a partition of $\c$ into four disjoint pieces each of size $\c$.
 Let $\langle {A}_{\alpha}: \alpha \in {T}_{0} \rangle$ be an enumeration of $\Pset(\omega)$.
 Let $\langle \langle {A}_{\alpha}, {X}_{\alpha} \rangle: \alpha \in {T}_{1} \rangle$ enumerate $\Pset(\omega) \times \Pset(\omega)$ in such a way that each element of $\Pset(\omega) \times \Pset(\omega)$ occurs $\c$ times on the list. 
 Let $\T$ = 
 \begin{align*}
 \left\{\langle P, f \rangle: P \subset {\[\omega\]}^{< \omega} \setminus \{0\} \ \text{and} \ f: P \rightarrow \omega \ \text{satisfy (1)-(2) of Lemma \ref{lem:localmap}} \right\}.
 \end{align*}
 Let $\langle \langle {A}_{\alpha}, {B}_{\alpha}, {P}_{\alpha}, {f}_{\alpha} \rangle: \alpha \in {T}_{2}\rangle$ enumerate $\Pset(\omega) \times \Pset(\omega) \times \T$ in such a way that every element of $\Pset(\omega) \times \Pset(\omega) \times \T$ occurs $\c$ times on the list.
 Build a decreasing sequence $\langle {p}_{\alpha}: \alpha < \c \rangle$ of conditions in ${\P}_{1}$ by induction as follows.
 Since ${\P}_{1}$ is non-empty choose an arbitrary ${p}_{0} \in {\P}_{1}$.
 If $\delta < \c$ is a limit ordinal, then by Lemmas \ref{lem:cofomegalimitcondition} and \ref{lem:unccoflimitcondition} there is a ${p}_{\delta} \in {\P}_{1}$ such that $\forall \alpha < \delta\[{p}_{\delta} \leq {p}_{\alpha}\]$.
 Now suppose $\delta = \alpha + 1$.
 If $\alpha \in {T}_{0}$, then use Lemma \ref{lem:main1} to find ${p}_{\delta} \in {\P}_{1}$ such that ${p}_{\delta} \leq {p}_{\alpha}$ and $\exists C \in {\A}_{{p}_{\delta}}\[{A}_{\alpha} \; {=}^{\ast} \; C\]$.
 If $\alpha \in {T}_{1}$ and ${A}_{\alpha} \in {\A}_{{p}_{\alpha}}$, then use Lemma \ref{lem:ultra} to find ${p}_{\delta} \in {\P}_{1}$ such that ${p}_{\delta} \leq {p}_{\alpha}$ and either ${X}_{\alpha} \in {\DD}_{{p}_{\delta}, {A}_{\alpha}}$ or $\omega \setminus {X}_{\alpha} \in {\DD}_{{p}_{\delta}, {A}_{\alpha}}$.
 If ${A}_{\alpha} \notin {\A}_{{p}_{\alpha}}$, then let ${p}_{\delta} = {p}_{\alpha}$.
 Next, suppose $\alpha \in {T}_{2}$, ${A}_{\alpha}, {B}_{\alpha} \in {\A}_{{p}_{\alpha}}$, and that ${B}_{\alpha} \; {\not\subset}^{\ast} \; {A}_{\alpha}$.
 Use Lemma \ref{lem:killingtukey} to find ${p}_{\delta} \in {\P}_{1}$ such that ${p}_{\delta} \leq {p}_{\alpha}$ and there exist ${X}_{\alpha} \in {\DD}_{{p}_{\delta}, {A}_{\alpha}}$ and ${Y}_{\alpha} \in {\DD}_{{p}_{\delta}, {B}_{\alpha}}$ such that $\forall s \in {P}_{\alpha}\[s \subset {X}_{\alpha} \implies {f}_{\alpha}(s) \notin {Y}_{\alpha}\]$.
 If $\alpha \in {T}_{2}$, but the other conditions are not satisfied, then let ${p}_{\delta} = {p}_{\alpha}$.
 Finally if $\alpha \in {T}_{3}$, then use Lemma \ref{lem:ppoint1} to find ${p}_{\delta} \in {\P}_{1}$ such that ${p}_{\delta} \leq {p}_{\alpha}$ and $\forall A \in {\A}_{{p}_{\delta}} \exists {Y}_{A, \alpha} \in {\DD}_{{p}_{\delta}, A} \forall X \in {\DD}_{{p}_{\delta}, A}\[ {Y}_{A, \alpha} \; {\subset}^{\ast} \; X\]$.
 This concludes the construction of $\langle {p}_{\alpha}: \alpha < \c \rangle$.
 
 Now define $\XX = {\bigcup}_{\alpha < \c} {{\A}_{{p}_{\alpha}}}$.
 It is easy to check that (1) holds.
 For any $A \in \XX$, let ${\alpha}_{A} = \min\{\alpha < \c: A \in {\A}_{{p}_{\alpha}}\}$.
 Define ${\U}_{A} = {\bigcup}_{{\alpha}_{A} \leq \alpha < \c} {{\DD}_{{p}_{\alpha}, A}}$.
 It is easy to check that ${\U}_{A}$ is a P-point.
 Next, say $A, B \in \XX$ with $A \; {\subset}^{\ast} \; B$.
 Let ${\alpha}_{A, B} = \max\{{\alpha}_{A}, {\alpha}_{B}\} < \c$.
 Define ${\pi}_{B, A} = {\pi}_{{p}_{{\alpha}_{A, B}}, B, A} \in \BS$.
 It is easy to check that if $X \in {\U}_{B}$, then ${\pi}_{B, A}'' X \in {\U}_{A}$.
 This implies that ${\U}_{A} \; {\leq}_{RK} \; {\U}_{B}$.
 Finally suppose $A, B \in \XX$ and that $B \; {\not\subset}^{\ast} \; A$.
 Suppose for a contradiction that ${\U}_{B} \; {\leq}_{T} \; {\U}_{A}$.
 Applying Lemma \ref{lem:localmap} with $\VV = {\U}_{A}$ and $\U = {\U}_{B}$ we can find $P \subset {\[\omega\]}^{< \omega} \setminus \{0\}$ and $f: P \rightarrow \omega$ satisfying (1)-(3) of Lemma \ref{lem:localmap}.
 There exists $\alpha \in {T}_{2}$ such that ${\alpha}_{A, B} \leq \alpha$ and ${A}_{\alpha} = A$, ${B}_{\alpha} = B$, ${P}_{\alpha} = P$, and ${f}_{\alpha} = f$.
 Let $\delta = \alpha + 1$.
 Then by construction there exist ${X}_{\alpha} \in {\DD}_{{p}_{\delta}, A} \subset {\U}_{A} = \VV$ and ${Y}_{\alpha} \in {\DD}_{{p}_{\delta}, B} \subset {\U}_{B} = \U$ such that $\forall s \in P\[s \subset {X}_{\alpha} \implies f(s) \notin {Y}_{\alpha}\]$, contradicting (3) of Lemma \ref{lem:localmap}.
 This concludes the proof of the theorem.
\end{proof}
\section{Remarks and open questions} \label{sec:questions}
 Under $\MA(\sigma-\textrm{centered})$ there are ${2}^{\c}$ P-points.
 Our results here leave open the question of which partial orders of size greater than $\c$ can be embedded into the P-points.
 As pointed out in the introduction, each P-point can have at most $\c$ predecessors with respect to ${\leq}_{RK}$ and also with respect to ${\leq}_{T}$.
 \begin{Def} \label{def:locallyc}
  A partial order $\langle X, < \rangle$ is said to be \emph{locally of size $\c$} if for each $x \in X$, $\lc \{x' \in X: x' \leq x \} \rc \leq \c$.
 \end{Def}
 \begin{Question} \label{q:blassremainder}
  Suppose $\MA(\sigma-\textrm(centered))$ holds.
  Let $\langle X, < \rangle$ be a partial order of size at most ${2}^{\c}$ that is locally of size $\c$.
  Does $\langle X, < \rangle$ embed into the class of P-points with respect to both the Rudin-Keisler and Tukey orders? 
 \end{Question}
 A positive answer to Question \ref{q:blassremainder} will give a complete solution to Blass' Question \ref{q:mainq}.
 It would say that anything that could possibly embed into the P-points does.
 As we have mentioned in the introduction, we are able to modify the techniques in this paper to deal with some specific cases of Question \ref{q:blassremainder}, like when $\langle X, < \rangle$ is the ordinal $\langle {\c}^{+}, \in \rangle$.
 However a general solution may require some new ideas.
\def\polhk#1{\setbox0=\hbox{#1}{\ooalign{\hidewidth
  \lower1.5ex\hbox{`}\hidewidth\crcr\unhbox0}}}
\providecommand{\bysame}{\leavevmode\hbox to3em{\hrulefill}\thinspace}
\providecommand{\MR}{\relax\ifhmode\unskip\space\fi MR }
\providecommand{\MRhref}[2]{%
  \href{http://www.ams.org/mathscinet-getitem?mr=#1}{#2}
}
\providecommand{\href}[2]{#2}

\end{document}